\documentclass[12pt]{article}

\usepackage{amsmath,amssymb,amsthm,graphicx,color,mathtools}
\usepackage{amsfonts}
\usepackage{verbatim}
\usepackage[left=1in,right=1in,top=1in,bottom=1in]{geometry}
\usepackage{cite}
\usepackage[british]{babel}
\usepackage{enumitem}
\usepackage{hyperref} 

\hypersetup{
    colorlinks=false,
    pdfborder={0 0 0},
}

\newtheorem{theorem}{Theorem}[section]

\newtheorem{proposition}[theorem]{Proposition}
\newtheorem{lemma}[theorem]{Lemma}
\newtheorem{conjecture}[theorem]{Conjecture}

\numberwithin{equation}{section}

\theoremstyle{definition}

\newtheorem{examplex}[theorem]{Example}

\newenvironment{example}
  {\pushQED{\qed}\examplex}
  {\popQED\endexamplex}

\theoremstyle{remark}
\newtheorem{remark}[theorem]{Remark}
\newtheorem{remarks}[theorem]{Remarks}
\newtheorem*{remark*}{Remark}

 \allowdisplaybreaks

\newcommand{\1}[1]{{\mathbf 1}{\{#1\}}}
\newcommand{\Z}{\mathbb{Z}}
\newcommand{\ZP}{\mathbb{Z}_+}
\newcommand{\R}{\mathbb{R}}

\newcommand{\N}{\mathbb{N}}

\DeclareMathOperator{\E}{\mathbb{E}}
\renewcommand{\Pr}{{\mathbb P}}

\DeclareMathOperator{\Var}{\mathbb{V}ar}

\DeclareMathOperator*{\supp}{supp}
 
\newcommand{\tra}{{\scalebox{0.6}{$\top$}}}

\newcommand{\eps}{\varepsilon}

\newcommand{\ba}{\mathbf{a}}
\newcommand{\bb}{\mathbf{b}}

\newcommand{\be}{\mathbf{e}}
\newcommand{\bk}{\mathbf{k}}
\newcommand{\bs}{\mathbf{s}}
\newcommand{\bt}{\mathbf{t}}
\newcommand{\bu}{\mathbf{u}}
\newcommand{\bx}{\mathbf{x}}
\newcommand{\by}{\mathbf{y}}
\newcommand{\bz}{\mathbf{z}}
\newcommand{\bmu}{\boldsymbol{\mu}}
\newcommand{\0}{\mathbf{0}}
\newcommand{\vo}{\mathbf{1}}
\newcommand{\re}{{\mathrm{e}}}

\newcommand{\ud}{{\mathrm d}}
\newcommand{\as}{\text{ a.s.}}

\newcommand{\cH}{{\mathcal{H}}}

\newcommand{\cL}{{\mathcal{L}}}
\newcommand{\cN}{{\mathcal{N}}}

\newcommand{\cX}{{\mathcal{X}}}
 
\newcommand{\tod}{\stackrel{d}{\longrightarrow}}
\newcommand{\eqd}{\stackrel{d}{=}}

\makeatletter
\def\namedlabel#1#2{\begingroup  
    (#2)%
    \def\@currentlabel{#2}%
    \phantomsection\label{#1}\endgroup
}
\makeatother

\title{On the centre of mass of a random walk}
\author{Chak Hei Lo \and Andrew R. Wade} 
\date{\today}

\begin{document}

\maketitle

\begin{abstract}
For a random walk $S_n$ on $\R^d$ we study the asymptotic behaviour of the 
associated centre of mass process $G_n = n^{-1} \sum_{i=1}^n S_i$. For lattice distributions we give conditions for a local limit theorem
to hold. We prove that if the increments of the walk have zero mean and finite second moment, $G_n$ is recurrent if $d=1$ and transient if $d \geq 2$. In the transient case we show that $G_n$ has 
a diffusive rate of escape.
These results extend work of Grill, who considered simple symmetric random walk.
We also give a class of  random walks with symmetric heavy-tailed increments for which $G_n$ is transient in $d=1$.
\end{abstract}

\smallskip
\noindent
{\em Keywords:} Random walk; centre of mass; barycentre; time-average; recurrence classification; local central limit theorem; rate of escape.

\smallskip
\noindent
{\em 2010 Mathematics Subject Classifications:} 60G50 (Primary) 60F05, 60J10 (Secondary).

\section{Introduction and main results}
\label{sec:results}

Let $d$ be a positive integer. Suppose that $X, X_1, X_2, \ldots$ is a sequence of i.i.d.~random variables on $\R^d$.
We consider the random walk $(S_n, n \in \ZP)$ in $\R^d$ defined by $S_0 := \0$ and $S_n :=\sum_{i=1}^{n}{X_i}$
($n \geq 1$). 
Our object of interest is the \emph{centre of mass process} $(G_n , n \in \ZP)$ corresponding to the random walk, defined by $G_0 := \0$ and $G_n := \frac{1}{n}\sum_{i=1}^{n}{S_i}$ ($n \ge 1$). The question of the asymptotic behaviour of $G_n$
was raised by P.~Erd\H os (see \cite{KG}).
We view vectors in $\R^d$ as column vectors throughout; $\0$ denotes the zero vector.
We write $\| \, \cdot \, \|$ for the Euclidean norm on $\R^d$.
Throughout we use the notation
\[ \bmu := \E X , ~~~ M := \E [ ( X - \bmu) (X - \bmu)^\tra ] \]
whenever the expectations exist; when defined, $M$ is a symmetric $d$ by $d$ matrix.

The strong law of large numbers for $S_n$ yields the following strong law for $G_n$, whose proof
can be found in Appendix~\ref{sec:appendix}.
\begin{proposition}
\label{prop:LLN}
Suppose that $\E \| X \| < \infty$. Then $n^{-1} G_n \to \frac{1}{2} \bmu$, a.s., as $n \to \infty$.
\end{proposition}
To go further we typically assume the following.
\begin{description}
\item
[\namedlabel{ass:basicd}{M}]
Suppose that $\E [ \| X \|^2 ] < \infty$ and $M$ is positive-definite. 
\end{description}
Note that
\begin{equation}
\label{eq:weighted-sum}
G_n = \sum_{i=1}^n \left( \frac{n-i+1}{n} \right) X_i .
\end{equation}
The representation~\eqref{eq:weighted-sum} leads via the
Lindeberg--Feller theorem for triangular arrays 
to the following central limit theorem; again, see Appendix~\ref{sec:appendix} for the proof.
We write `$\tod$' for convergence in distribution, and $\cN_d ( \mathbf{m} , \Sigma )$
for a $d$-dimensional normal random variable with mean $\mathbf{m}$ and covariance $\Sigma$. 
\begin{proposition}
\label{prop:CLT}
If~\eqref{ass:basicd} holds, then, as $n \to \infty$,
\[ n^{-1/2} \left( G_n - \frac{n}{2} \bmu \right) \tod \cN_d ( \0 , M/ 3 ) .\]
\end{proposition}

Our first main result is a \emph{local} central limit theorem. We assume that $X$ has a non-degenerate $d$-dimensional lattice distribution.
Thus (see \cite[Ch.~5]{BR}) there is a unique minimal subgroup $L := H \Z^d$ of $\R^d$, where $H$ is a $d$ by $d$ matrix, 
such that $\Pr ( X \in \bb + L ) =1$ for some $\bb \in \R^d$,
with the property that if $\Pr ( X \in \bx + L' ) = 1$ for some closed subgroup $L'$ and $\bx \in \R^d$, then $L \subseteq L'$,
and with $h := | \det H | \in (0, \infty)$. In other words, we make the following assumption.
\begin{description}
\item
[\namedlabel{ass:basicd2}{L}]
Suppose that the minimal subgroup associated with $X$ is $L := H \Z^d$ with $h := | \det H | > 0$.
\end{description}
See Appendix~\ref{sec:char-fn} for background on lattice distributions.
Equivalent conditions to~\eqref{ass:basicd2}
can be formulated in terms of the characteristic function of $X$ or in terms of the maximality of $h$: see Lemma~\ref{lem:equivalent} below.
Note that there may be many matrices $H$ for which $H \Z^d$ is equal to (unique) $L$, but for all of these $| \det H|$ is the same.
Also note that symmetric simple random walk (SSRW) does not satisfy~\eqref{ass:basicd2} with the obvious choice $H = I$ (the identity), but does
satisfy~\eqref{ass:basicd2} for an $H$ with $h=2$, the maximal $| \det H|$ 
for which $\Pr ( X \in \bx + H \Z^d ) =1$ for some $\bx \in \R^d$: see Section~\ref{sec:examples} and Appendix~\ref{sec:char-fn} for details. 

Notice that $\Pr(X \in \bb +H \Z^d)=1$ implies $\Pr(S_n \in n\bb + H \Z^d)=1$, which shows that
 $\Pr(n^{-1/2} G_n \in \cL_n ) = 1$, where we define 
\[ \cL_n := \left\{ n^{-3/2} \left( \tfrac{1}{2}n(n+1)\bb + H \Z^d \right) \right\}. \]
For $\bx \in \R^d$, define
$p_n ( \bx )  := \Pr ( n^{-1/2} G_n= \bx )$,  and
\begin{align}
\label{eq:ndef}
\nu (\bx) & : = \frac{\exp \{- \tfrac{3}{2} \bx^\tra M^{-1}\bx  \} }{(2\pi)^{d/2} \sqrt{\det (M/3)}} ,
\end{align}
the density of $\cN_d ( \0, M/3)$.

Here is our local limit theorem.
\begin{theorem}
\label{thm:LCLT}
Suppose that~\eqref{ass:basicd2} and~\eqref{ass:basicd} hold. Then we have 
\begin{equation}
\lim_{n \to \infty} \sup_{\bx \in \cL_n} \left| \frac{n^{3d/ 2}}{h} p_n(\bx)- 
\nu \left(\bx - \frac{(n+1)}{2n^{1/2}} \bmu \right) \right| = 0. 
\label{eq:lclt}
\end{equation}
\end{theorem}

\begin{remarks}
(i)
In the case $d=1$, versions of Theorem~\ref{thm:LCLT} are given in~\cite[Lemma~4.3]{MP}, in~\cite[Proposition~2.3]{CD},
and in~\cite[Lemma~1]{KG};
the latter result deals only with the special case of SSRW
and only bounds $p_n (\bx)$ up to constant factors. See Section~\ref{sec:examples} for a demonstration
that our assumptions are indeed satisfied by SSRW on $\Z^d$ for appropriate choice of $H$ with $h=2$.
The proof in \cite{MP} is a sketch, and the statement that ``it is enough to apply the usual analytical methods'' \cite[p.~515]{MP} does not quite tell the whole story, even in the one-dimensional case.
The papers~\cite{CD,KG,MP} also give bivariate local limit theorems for $(S_n, G_n)$ (in the case $d=1$). Related results can be found in \cite[Theorem~4.2]{DH} and~\cite{DKW}.

(ii)
If $Z_n := S_n - G_n$, then~\eqref{eq:weighted-sum} shows that $Z_{n+1} = \frac{n}{n+1} \sum_{i=1}^n (i/n) X_{i+1}$,
which implies that $Z_{n+1} \eqd \frac{n}{n+1} G_n$, where
 `$\eqd$' stands for equality in distribution.
 Thus Theorem~\ref{thm:LCLT} also yields
a local limit theorem for $Z_n$. However, the \emph{processes} $Z_n$ and $G_n$ may behave
very differently: see \cite[Remark 1.1]{bary}.
\end{remarks}

We turn to the almost-sure asymptotic behaviour of $G_n$.
First we have a recurrence result for $d=1$ that does not require the lattice assumption; in the case
of SSRW the fact that $G_n$ returns i.o.~(infinitely often) to a neighbourhood of the origin is due to Grill~\cite[Theorem~1]{KG}.
\begin{theorem}
\label{thm:classification}
Suppose that $d=1$ and that either of the following two conditions
holds.
\begin{itemize}
\item[(i)] Suppose that $\E | X  | \in (0, \infty)$ and $X \eqd - X$.
\item[(ii)] Suppose that~\eqref{ass:basicd} holds and that $\E X = 0$.
\end{itemize}
Then  $\liminf_{n \to \infty} G_n= -\infty$, $\limsup_{n \to \infty} G_n= +\infty$, and
 $\liminf_{n \to \infty} | G_n -x  | = 0$ for any $x \in \R$.
\end{theorem} 

In contrast to Theorem~\ref{thm:classification}, we will show that in the case where
 $\E |X| = \infty$, $G_n$ may be transient. The condition we assume is as follows.

\begin{description}
\item
[\namedlabel{ass:stable}{S}]
Suppose that $X \eqd -X$ and $X$ is in the domain of normal attraction of a symmetric $\alpha$-stable  distribution with $\alpha \in (0,1)$.
\end{description}

\begin{theorem}
\label{thm:stable-transience}
Suppose that~$d=1$ and~\eqref{ass:basicd2} holds, i.e., $\Pr(X \in b + h \Z)=1$ for $b \in \R$ and $h >0$ maximal. Suppose also that~\eqref{ass:stable} holds. 
Then 
$\liminf_{n \to \infty} G_n= -\infty$, $\limsup_{n \to \infty} G_n= +\infty$, and
$\lim_{n \to \infty} | G_n | = \infty$.
\end{theorem}
\begin{remark}
The transience here fails in the natural continuum version of this model. The analogous continuum model, a symmetric $\alpha$-stable L\'evy process for $\alpha \in (0,1)$, $s_t$,
has centre of mass $g_t = \frac{1}{t} \int_0^t s_u \ud u$, and it is surely true that $g_t$ again changes sign i.o., but in this case continuity of $g_t$
implies that $g_t =0$ i.o.
\end{remark}

We have the following transience result in dimensions greater than one.
In particular, Theorem~\ref{thm:classification2} says that $\lim_{n \to \infty} \|G_n \| = +\infty$, a.s., and gives a diffusive rate of escape;
in the case
of SSRW the result is due to Grill~\cite[Theorem~1]{KG}. 

\begin{theorem}
\label{thm:classification2}
Suppose that $d \geq 2$ and that~\eqref{ass:basicd2} and~\eqref{ass:basicd} hold, and that $\bmu = \0$. Then 
\[ \lim_{n \to \infty} \frac{ \log \| G_n \|}{\log n} = \frac{1}{2}, \as \]
\end{theorem}

Obtaining necessary and sufficient conditions for recurrence and transience of $G_n$
is an open problem. For $d \geq 2$, we expect that $G_n$ is always `at least as transient'
as the situation in Theorem~\ref{thm:classification2}:

\begin{conjecture}
Suppose that $\supp X$ is not contained in a one-dimensional subspace of $\R^d$. Then
\[ \liminf_{n \to \infty} \frac{ \log \| G_n \|}{\log n} \geq \frac{1}{2}, \as \]
\end{conjecture}

 Section~\ref{sec:examples} verifies our main assumptions for a couple of simple examples.
The proof of Theorem~\ref{thm:LCLT} is given in Section~\ref{sec:LCLT-proof}.
The proof of Theorem~\ref{thm:classification} uses Proposition~\ref{prop:CLT},
some observations following from the Hewitt--Savage zero--one law, and the fact that in the
case where $\E X = 0$ oscillating behaviour is sufficient for $\liminf_{n \to \infty} | G_n - x| = 0$: see Section~\ref{sec:one-dimension}. 
The proof of Theorem~\ref{thm:stable-transience} uses another local limit theorem (Theorem~\ref{thm:SLCLT})
and is also presented in Section~\ref{sec:one-dimension}. 
The proof of Theorem~\ref{thm:classification2}
relies on Theorem~\ref{thm:LCLT}: see Section~\ref{sec:transience}. Appendix~\ref{sec:char-fn} collects auxiliary results on lattice distributions and characteristic functions that we need for the proofs of our local limit theorems. 
For completeness we include the proofs of Propositions~\ref{prop:LLN} and~\ref{prop:CLT} in Appendix~\ref{sec:appendix}.

\section{Examples}
\label{sec:examples}

We use the notation $\varphi (\bt) := \E [ \re^{i \bt^\tra X} ]$ for the characteristic function of $X$.
Set 
$U := \{ \bt \in \R^d : | \varphi (\bt ) | = 1 \}$, and
given an  invertible $d$ by $d$ matrix $H$, set $S_H := 2 \pi ( H^\tra)^{-1} \Z^d$.

\begin{example}[Lazy SSRW on $\Z^d$]
Let $\be_1, \ldots, \be_d$ be the standard orthonormal basis vectors of $\R^d$, and suppose that
$\Pr ( X = \be_i ) = \Pr ( X = - \be_i ) = \frac{1}{4d}$ for all $i$, and $\Pr ( X = \0) = \frac{1}{2}$.
Then for $\bb = \0$ and $H = I$,  the $d$ by $d$ identity matrix, we have $\Pr ( X \in \Z^d ) =1$.
To verify that $L = \Z^d$ is minimal, it is sufficient (see Lemma~\ref{lem:equivalent}) to check
that $U = S_H = 2 \pi \Z^d$.
If $\bt = ( t_1, \ldots, t_d ) \in \R^d$,
\[ \varphi (\bt) = \frac{1}{2} + \frac{1}{4d} \sum_{j=1}^d \left( \re^{it_j} + \re^{-it_j} \right) = \frac{1}{2} + \frac{1}{2d} \sum_{j=1}^d \cos  t_j .\]
Thus $\bt \in U$ if and only if $\cos t_j = 1$ for all $j$, i.e., $U = 2 \pi \Z^d = S_H$, as required.
Note that we could alternatively use the bound in Lemma~\ref{lem:h-bounded} to check that $h=1$ is maximal.
\end{example}

\begin{example}[SSRW on $\Z^d$]
Suppose that $\Pr ( X = \be_i ) = \Pr ( X = - \be_i ) = \frac{1}{2d}$ for all $i$.
For SSRW the construction of $H$ for which~\eqref{ass:basicd2} holds is non-trivial.
For $d=1$, we take $b = -1$ and $h = 2$. In general $d \geq 2$, we take $H = ( h_{ij} )$ and $\bb = (b_i)$ defined as follows.
If $d = 2n - 1$ for $n \ge 2, n \in \Z$, we take
\begin{align*}
b_i &= -1 \quad \text{for all } i = 1, 2, \ldots , d; \\
h_{ij} &=
\begin{cases}
1 & \text{if } i-j \equiv 0 \text{ or } n \pmod{2n-1} ,\\
0 & \text{otherwise}.
\end{cases}
\end{align*}
If $d = 2n $ for $n \ge 1, n \in \Z$, we take
\begin{align*}
b_i &=
\begin{cases}
0 & \text{if } i=2n ,\\
-1 & \text{otherwise};
\end{cases}
\\
h_{ij} &=
\begin{cases}
-1 & \text{if } (i,j)= (2n, 1), \\
1 & \text{if } j-i \equiv 0 \text{ or } 1 \pmod{2n} \text{ and } (i,j) \neq (2n, 1),\\
0 & \text{otherwise}.
\end{cases}
\end{align*}
For example, for $d=2$ we have
\begin{equation*}
\bb =
\begin{pmatrix}
-1 \\
0
\end{pmatrix} 
\quad \text{and} \quad H=
\begin{pmatrix}
1 & 1 \\
-1 & 1
\end{pmatrix}.
\end{equation*}
For $d=3$, we have
\begin{equation*}
\bb =
\begin{pmatrix}
-1 \\
-1 \\
-1 
\end{pmatrix} 
\quad \text{and} \quad H=
\begin{pmatrix}
1 & 1 & 0 \\
0 & 1 & 1 \\
1 & 0 & 1
\end{pmatrix}.
\end{equation*}
For $d=4$, we have 
\begin{equation*}
\bb =
\begin{pmatrix}
-1 \\
-1 \\
-1 \\
0
\end{pmatrix} 
\quad \text{and} \quad H=
\begin{pmatrix}
1 & 1 & 0 & 0 \\
0 & 1 & 1 & 0 \\
0 & 0 & 1 & 1 \\
-1 & 0 & 0 & 1
\end{pmatrix}.
\end{equation*}
For $d=5$, we have
\begin{equation*}
\bb =
\begin{pmatrix}
-1 \\
-1 \\
-1 \\
-1 \\
-1
\end{pmatrix} 
\quad \text{and} \quad H=
\begin{pmatrix}
1 & 0 & 1 & 0 & 0 \\
0 & 1 & 0 & 1 & 0 \\
0 & 0 & 1 & 0 & 1 \\
1 & 0 & 0 & 1 & 0 \\
0 & 1 & 0 & 0 & 1
\end{pmatrix},
\end{equation*}
and so on. Note that $h=2$ for all such $H$. It is elementary to verify that
$\Pr ( X \in \bb + H \{ 0, 1 \}^d ) = 1$. It suffices to check
that $H^{-1} ( \bx - \bb ) \in \{0,1\}^d$ for any $\bx = \pm \be_i$.
For example, in the case $d=2n-1$ we have that 
$H^{-1}$ has elements $h_{ij}^{-1}$ given by
\[ h^{-1}_{ij} = \begin{cases} \frac{1}{2} & \text{if } i - j = 0, 1, \ldots, n-1 \pmod{2n-1} , \\
- \frac{1}{2} &\text{otherwise} , \end{cases} \]
and then one checks that, for example, $H^{-1} ( \be_i - \bb) = \ba$ where $\ba$ has all components zero apart from
$a_i = \cdots = a_{i+n-1} = 1$ (for $i \leq n$). The other cases are similar. 

We show that~\eqref{ass:basicd2} holds for SSRW with this choice of $H$,
by checking  (see Lemma~\ref{lem:equivalent}) that $U = S_H$. Since Lemma~\ref{lem:property} shows that
$S_H \subseteq U$, it suffices to show that $U \subseteq S_H$.
For SSRW on $\Z^d$, 
if $\bt = ( t_1, \ldots, t_d ) \in \R^d$,
\[ \varphi ( \bt ) = \frac{1}{2d} \sum_{j=1}^d \left( \re^{it_j} + \re^{-it_j} \right) = \frac{1}{d} \sum_{j=1}^d \cos  t_j .\]
So $\bt \in U$ if and only if
$| \sum_{j=1}^d \cos t_j | = d$, which occurs if and only if either (i)
$\cos t_j = 1$ for all $j$, or (ii) $\cos t_j = -1$ for all $j$.
Case (i) is equivalent to 
$\bt \in 2 \pi \Z^d$ and case (ii) is equivalent to 
$\bt \in \pi \vo + 2 \pi \Z^d$, where $\vo$ is the vector of all $1$s.
Hence
\[ U = ( 2 \pi \Z^d ) \cup (\pi \vo + 2 \pi \Z^d ) .\]

Consider $\bx \in U$. Then for some $\ba \in \Z^d$, either
 (i) $\bx = 2 \pi \ba$, or (ii) $\bx = \pi \vo + 2 \pi \ba$.
In case (i), let $\bz = H^\tra \ba$; then since all entries in $H$
are integers, we have
$\bz \in \Z^d$ and
$2\pi ( H^\tra )^{-1} \bz = 2 \pi \ba = \bx$, so $\bx \in S_H$.
In case (ii), let $\bz = H^\tra ( \frac{1}{2} \vo + \ba )$.
Note that if $d$ is odd then $\frac{1}{2} H^\tra \vo = \vo$ while if
$d$ is even, $\frac{1}{2} H^\tra \vo = (0,1,1,\ldots,1)^\tra$;
in any case it follows that $\bz \in \Z^d$.
Then  
$2 \pi ( H^\tra )^{-1} \bz = \pi \vo + 2 \pi \ba = \bx$, so $\bx \in S_H$.
Thus $U \subseteq S_H$.
\end{example}

\section{Local central limit theorem}
\label{sec:LCLT-proof}

This section is devoted to the proof of Theorem~\ref{thm:LCLT}. 
The outline of the proof mirrors the standard Fourier-analytic 
proof of the local central limit theorem for
the random walk: compare e.g.~\cite[Ch.~9]{GK}, \cite[\S 3.5]{RD}, or \cite[Ch.~4]{IL}
for the one-dimensional case, and~\cite[\S\S 2.2--2.3]{LL} for the case of walks on $\Z^d$. The details
of the proof require some extra effort, however.

First we show that it suffices to establish Theorem~\ref{thm:LCLT} in the case
where $\bb = \0$ and $H = I$ (the identity). To see this,
suppose that $X \in \bb + H \mathbb{Z}^d$ and set $\tilde X = H^{-1} ( X - \bb )$.
Then $\tilde X \in \Z^d$. By linearity of expectation, we have 
\[ \tilde \bmu := \E \tilde X = H^{-1} ( \bmu - \bb), \text{ and }
\tilde M := \E [ ( \tilde X - \tilde \bmu) ( \tilde X - \tilde \bmu)^\tra ] = H^{-1} M (H^{-1} )^\tra. \]
Note that $(H^{-1})^\tra$ is nonsingular, so $(H^{-1})^\tra \bx \neq \0$ for all $\bx \neq \0$.
Hence for $\bx \neq \0$,
$\bx^\tra \tilde M \bx = \by^\tra M \by$ where $\by = (H^{-1})^\tra \bx \neq \0$, so that
since $M$ is positive definite we have $\bx^\tra \tilde M \bx > 0$; hence $\tilde M$ is also positive definite.
Also, $\tilde S_n := \sum_{i=1}^n \tilde X_i = H^{-1} (S_n - n\bb)$ and
$\tilde G_n := n^{-1} \sum_{i=1}^n \tilde S_i = H^{-1} ( G_n - \frac{n+1}{2} \bb )$.
The assumption
that $H \Z^d$ is minimal for $X$ implies that $\Z^d$ is minimal for $\tilde X$.
Thus the process defined by $\tilde X$ satisfies the hypotheses of Theorem~\ref{thm:LCLT} in the case
where $\bb = \0$ and $H = I$, with mean $\tilde \bmu$ and covariance $\tilde M$,
and that result yields
\begin{equation}
\label{eq:reduction1}
 \lim_{n \to \infty} \sup_{\bx \in n^{-3/2} \Z^d} \left| n^{3d/2} \Pr ( n^{-1/2} \tilde G_n = \bx ) - \tilde \nu \left( \bx - \frac{(n+1)}{2 n^{1/2}} \tilde \bmu \right) \right| = 0,\end{equation}
where
\[ \tilde \nu ( \bz ) := \frac{( \det \tilde M / 3 )^{-1/2}}{(2 \pi )^{d/2}}  \exp \left\{ - \frac{3}{2} \bz^\tra \tilde M^{-1} \bz \right\} .\]
But
\[ \Pr ( n^{-1/2} \tilde G_n = \bx ) = \Pr \left( n^{-1/2} G_n = \frac{(n+1)}{2 n^{1/2}} \bb + H \bx \right) = \Pr (n^{-1/2} G_n = \by ) \]
where $\by = \frac{(n+1)}{2 n^{1/2}} \bb + H \bx$
so $\by \in n^{-3/2} ( \frac{1}{2} n(n+1) \bb + H \Z^d )$.
Also,
\begin{align*}   \bx - \frac{(n+1)}{2 n^{1/2}} \tilde \bmu   & = 
\left( H^{-1} \by - \frac{(n+1)}{2n^{1/2}} H^{-1} \bb \right) -\frac{(n+1)}{2n^{1/2}} H^{-1} ( \bmu -\bb) \\
& =  H^{-1} \by -  \frac{(n+1)}{2n^{1/2}} H^{-1}  \bmu .
\end{align*}
Hence, since $\tilde M^{-1} = H^\tra M^{-1} H$ and $\det \tilde M = h^{-2} \det M$,
\begin{align*}
\tilde \nu \left( \bx - \frac{(n+1)}{2 n^{1/2}} \tilde \bmu \right)  & =\frac{( \det \tilde M / 3 )^{-1/2}}{(2 \pi )^{d/2}} 
 \exp \left\{ - \frac{3}{2} \left(  \by - \frac{(n+1)}{2n^{1/2}} \bmu \right)^\tra M^{-1}
\left(  \by - \frac{(n+1)}{2 n^{1/2}} \bmu \right) \right\} \\
& = h \nu \left(  \by - \frac{(n+1)}{2 n^{1/2}} \bmu \right) .
 \end{align*}
It follows that~\eqref{eq:reduction1} is equivalent to
\[ \lim_{n \to \infty} \sup_{\by \in n^{-3/2} ( \frac{1}{2} n(n+1) \bb + H \Z^d ) } \left| \frac{n^{3d/2}}{h} \Pr ( n^{-1/2}  G_n = \by ) -  \nu \left(  \by - \frac{(n+1)}{2 n^{1/2}} \bmu \right) \right| = 0 ,\]
which is the general statement of Theorem~\ref{thm:LCLT}. Thus for the remainder of this section we  suppose that $\bb = \0$ and $H = I$; hence $\cL_n = n^{-3/2} \Z^d$.

Let $Y_n := \sum_{i=1}^{n}{S_i}$ and thus $G_n = Y_n/n$. 
Recall that $\varphi$ denotes the characteristic function (ch.f.) of $X$,
and let $\Phi_n$ be the ch.f.~of  $n^{-3/2}Y_n$, i.e., for $\bt \in \R^d$,
\[ \varphi (\bt) := \E \re^{i\bt^\tra X}, \text{ and } \Phi_n(\bt) := \E \re^{in^{-3/2}\bt^\tra Y_n}. \]
Denoting the smallest eigenvalue of $M$ by $\lambda_{\rm min} (M)$ 
and writing  $\hat \bt := \bt / \| \bt \|$ for $\bt \neq \0$,
we have that
 \begin{equation}
\label{eq:positive-definite}
\inf_{\bt \neq \0} \hat \bt^\tra M \hat \bt = \lambda_{\rm min} (M) >0, 
\end{equation}
since $\lambda_{\rm min} (M)$ is an eigenvalue of a positive-definite matrix
under assumption~\eqref{ass:basicd}.
Define
\begin{equation}
\label{f-def}
 f_n ( \bt ) := \exp \left\{ \frac{i (n+1) \bt^\tra \bmu}{2 n^{1/2} } - \frac{ \bt^\tra M \bt}{6} \right\} .
\end{equation}
For $\ell \in (0,\infty)$ set $R_1 := [-\ell,\ell]^d$. 
Our starting point for the proof of the local limit theorem is the following.

\begin{lemma}
\label{lem:lclt-first}
Suppose that~\eqref{ass:basicd} holds and that $\Pr ( X \in \Z^d ) = 1$.
 Then, for any $\ell \in (0,\infty)$,
\begin{align*}
\sup_{\bx \in \cL_n} \left|  n^{3d/2}  p_n(\bx)-\nu \left(\bx - \frac{(n+1)}{2n^{1/2}} \bmu \right) \right| &
\le \int_{R_1}
 \left| D_n(\bt) \right| \ud \bt  +  
\int_{R(n) \setminus R_1} | \Phi_n (\bt) | \ud \bt \\
& {} \qquad {} + \int_{\R^d \setminus R_1} \exp \left\{ -\frac{ \lambda_{\rm min} (M)}{6}  \| \bt \|^2 \right\} \ud \bt ,
\end{align*}
where $R(n) := [-\pi n^{3/2} , \pi n^{3/2} ]^{d}$
 and $D_n(\bt):= \Phi_n(\bt) - f_n (\bt)$. 
\end{lemma}
\begin{proof}
For a random variable $W \in \Z^d$, by the 
inversion formula for the characteristic function (see e.g.~\cite[Corollary~2.2.3, p.~29]{LL})
we have that 
\begin{equation}
\Pr(W = \by ) = \frac{1}{(2\pi )^{d}} \int_{[-\pi,\pi]^d} \re^{-i\bu^\tra \by} \E \bigl[ \re^{i\bu^\tra W} \bigr] \ud \bu, 
\label{eq:formula}
\end{equation}
for $\by \in \Z^d$. Now we have for $\bx \in \cL_n$,
$
p_n(\bx) = \Pr ( Y_n= n^{3/2} \bx )$,
so applying~\eqref{eq:formula} with $W=Y_n \in \Z^d$, we get for $\bx \in \cL_n$ that
\begin{align*}
p_n(\bx )  = 
\frac{1}{(2\pi )^{d}} 
\int_{[-\pi,\pi]^d} \re^{ -i n^{3/2} \bu^\tra \bx } \E\bigl[ \re^{i\bu^\tra Y_n }\bigr] \ud \bu.
\end{align*}
Using the substitution $\bu =n^{-3/2} \bt$, we obtain
\begin{equation}
 n^{3d/2}  p_n(\bx ) = \frac{1}{(2\pi )^{d}} \int_{[-\pi n^{3/2},\pi n^{3/2}]^d} \re^{-i\bt^\tra \bx} \Phi_n(\bt)\ud \bt. 
\label{eq:cal11} 
\end{equation}
On the other hand, since the probability density 
$\nu (\bx - \frac{(n+1)}{2n^{1/2}} \bmu)$, with $\nu$ as defined at~\eqref{eq:ndef}, corresponds to the ch.f.~$f_n (\bt)$
as defined at~\eqref{f-def}, the inversion formula for densities yields
\begin{equation}
\nu \left( \bx - \frac{(n+1)}{2n^{1/2}} \bmu \right)=\frac{1}{(2\pi)^d}\int_{\R^d} \re^{-i\bt^\tra \bx} f_n (\bt) \ud \bt, 
\label{eq:cal13}
\end{equation}
for $\bx \in \R^d$. Now we subtract~\eqref{eq:cal13} from~\eqref{eq:cal11} to get
\begin{align*}
 n^{3d/2}  p_n(\bx)- \nu \left( \bx - \frac{(n+1)}{2n^{1/2}} \bmu \right)
& = \frac{1}{(2\pi)^d} \int_{R_1} \re^{-i\bt^\tra \bx} D_n (\bt) \ud \bt 
+ \frac{1}{(2\pi)^d} \int_{R(n) \setminus R_1} \re^{-i\bt^\tra \bx}  \Phi_n (\bt ) \ud \bt \\
& {} \qquad {} - \frac{1}{(2\pi)^d} \int_{\R^d \setminus R_1} \re^{-i\bt^\tra \bx} f_n (\bt) \ud \bt .\end{align*}
Thus, by~\eqref{f-def} and the triangle inequality with the estimates $\pi>1$ and $|\re^{-i \bt^\tra \bx}| \le 1$, 
\begin{align*}
 \sup_{\bx \in \cL_n} \left|  n^{3d/2}  p_n(\bx)- \nu \left( \bx - \frac{(n+1)}{2n^{1/2}} \bmu \right) \right|  
& \le \int_{R_1} 
 \left| D_n(\bt) \right| \ud \bt  +  \int_{R(n) \setminus R_1} | \Phi_n (\bt) | \ud \bt \\
& {} \qquad{} + \int_{\R^d \setminus R_1} \exp \left\{ -\frac{\bt^\tra M \bt}{6} \right\} \ud \bt ,
 \end{align*}
which with~\eqref{eq:positive-definite} yields the statement in the lemma.
\end{proof}

To prove Theorem~\ref{thm:LCLT} we must show that the right-hand side of the expression in Lemma~\ref{lem:lclt-first} approaches $0$ when $n \to \infty$.
To do so, we bound $| \Phi_n (\bt)|$ and $|D_n (\bt)|$ for appropriate regions of $\bt$.
Observing that $Y_n= \sum_{i=1}^n S_i=\sum_{j=1}^n (n-j+1)X_j$, we see
\[
\Phi_n(\bt) = \E\left[\exp\left\{ i n^{-3/2} \bt^\tra Y_n\right\} \right] = \E\left[\exp\left\{ in^{-3/2}\sum_{j=1}^{n}(n-j+1)\bt^\tra X_j\right\} \right].
\]
For fixed $n$, $\sum_{j=1}^{n}(n-j+1) \bt^\tra X_j \eqd \sum_{j=1}^{n}j \bt^\tra X_j$, so that, by independence,  
\[
\Phi_n(\bt) = \E\left[\exp\left\{ in^{-3/2} \sum_{j=1}^{n}j \bt^\tra X_j\right\} \right] = \prod_{j=1}^{n}\E\left[\exp\left\{ i n^{-3/2}j \bt^\tra X_j \right\} \right] .
\]
Hence we conclude that for $\bt \in \R^d$,
\begin{equation}
\label{eq:Phi-phi}
\Phi_n( \bt) 
= \prod_{j=1}^{n} 
\varphi(n^{-3/2}j \bt ).
\end{equation}
To study $\Phi_n$ we 
require certain characteristic function estimates, presented in Appendix~\ref{sec:char-fn}.

Recall that $R(n) = [-\pi n^{3/2} , \pi n^{3/2} ]^{d}$ and $R_1 = [ -\ell, \ell]^d$. Also define the regions
\begin{align*}
R_2(n) & :=  [-\delta \sqrt{n}, \delta \sqrt{n}]^d \setminus R_1 \\
R_3(n) & :=  [-\pi \sqrt{n} , \pi \sqrt{n} ]^d \setminus (R_1 \cup R_2 (n) ) \\
R_4(n) & := R (n) \setminus (R_1 \cup R_2 (n) \cup R_3 (n))
\end{align*}
where the  constant $\delta \in (0, \pi)$ will be chosen later.
We denote the corresponding integrals by 
\[ I_1 (n) := \int_{R_1} | D_n (\bt) | \ud \bt, \text{ and } I_k(n) :=\int_{R_k(n)} \left| \Phi_n(\bt) \right| \ud \bt, \text{ for } k \in \{2,3,4\}. \] 
\begin{lemma}
\label{lem:parts1}
For $\delta >0$ sufficiently small, the following statements are true.
\begin{itemize}
\item[(i)] For any $\ell \in (0,\infty)$, $\lim_{n \to \infty} I_1(n)=0$,
\item[(ii)] $\lim_{\ell \to \infty} \sup_n I_2(n)=0$, 
\item[(iii)] $\lim_{n \to \infty} I_3(n)=0$,
\item[(iv)] $\lim_{n \to \infty} I_4(n)=0$.
\end{itemize}
\end{lemma}
We will combine all the estimates at the end of the argument. 
\begin{proof}[Proof of Lemma~\ref{lem:parts1}]
First we aim to show that
\begin{align}
\lim_{n \to \infty} \sup_{\bt \in R_1} |D_n(\bt)| =0 .
\label{eq:sup1}
\end{align} 
Since $\E X =\bmu$ and $\E[(X -\bmu) (X-\bmu)^\tra]=M$, we have
$\E [ X X^\tra ] = M + \bmu \bmu^\tra$,
so that
Lemma~\ref{lem:estchf} implies, uniformly over 
$\bt \in R_1$, as $n \to \infty$, 
\begin{align*}
\prod_{j=1}^{n} \varphi(n^{-3/2}j\bt) 
&= \exp \left\{ \sum_{j=1}^{n} \log \left[ 1 + A( n, j, \bt ) + o(n^{-1}) \right] \right\},
\end{align*}
where 
\begin{equation}
\label{A-def}
 A (n, j, \bt) :=   i n^{-3/2} j   \bt^\tra \bmu - \frac{1}{2} n^{-3} j^2 \bt^\tra (M + \bmu \bmu^\tra) \bt  .\end{equation}
Taylor's theorem for a complex variable shows that for a constant $C < \infty$,
\begin{equation}
\label{eq:complex-taylor}
\left| \log ( 1+ z ) - \left( z - \frac{z^2}{2} \right) \right| \leq C | z |^3 ,
\end{equation}
for all $z$ with $|z| < 1/2$.
Note from~\eqref{A-def} that
\begin{equation}
\label{A-sq}
 A (n, j, \bt)^2 = - n^{-3} j^2 \bt^\tra \bmu \bmu^\tra \bt + \Delta_0 (n,j,\bt) , \end{equation}
where $\max_{1 \leq j \leq n} \sup_{\bt \in R_1} | \Delta_0 (n ,j, \bt) | = O ( n^{-3/2} )$.
Then, by~\eqref{eq:Phi-phi}, \eqref{eq:complex-taylor}, \eqref{A-sq},
and the fact that
$\max_{1 \leq j \leq n} \sup_{\bt \in R_1} | A (n ,j, \bt) | = O ( n^{-1/2} )$, 
 it follows that
\[
\Phi_n (\bt) = \exp \left\{ \sum_{j=1}^{n} \left(i n^{-3/2} j   \bt^\tra \bmu -\frac{1}{2}n^{-3}j^2 \bt^\tra M \bt \right) +  \Delta_0 ( n ,\bt)   \right\}, 
\]
where $\sup_{\bt \in R_1} | \Delta_0 ( n ,\bt) | \to 0$.
 Elementary algebra gives $\sum_{j=1}^{n} j  = \frac{1}{2}n(n+1)$ and $\sum_{j=1}^{n} j^2 = \frac{1}{6}n(n+1)(2n+1)$,
 so we obtain the estimate 
\begin{align}
\nonumber
\Phi_n(\bt )= \exp \left\{ \frac{i (n+1) \bt^\tra \bmu}{2n^{1/2}} -\frac{\bt^\tra M \bt}{6} + \Delta_1 ( n ,\bt) \right\},
\end{align} 
where $\sup_{\bt \in R_1} | \Delta_1 ( n ,\bt) | \to 0$ as $n \to \infty$.
Hence, by~\eqref{f-def},
\[ | D_n (\bt ) | = | \Phi_n (\bt) - f_n (\bt) | \leq \left| 1 - \exp \{ \Delta_1 (n, \bt) \} \right| ,\]
which establishes~\eqref{eq:sup1} and proves part (i) of the lemma. 

Fix $\eps \in (0, \lambda_{\rm min} (M)/ 12 )$.
Suppose that $\bt \in [0, \delta n^{1/2}]^d$.
Then for $1 \leq j \leq n$, we have $\| n^{-3/2} j \bt \| \leq \delta d^{1/2}$. 
Thus, from  Lemma~\ref{lem:estchf},
\begin{align*}
\varphi ( n^{-3/2} j \bt ) & =  1 + A (n,j ,\bt) + \Delta_1 (n,j,\bt) ,
\end{align*}
where $A (n, j ,\bt)$ is as defined at~\eqref{A-def}, and
$|\Delta_1 ( n, j, \bt ) | \leq \eps n^{-1} \| \bt \|^2$
for all $\bt \in  [0, \delta n^{1/2}]^d$ and $\delta$ sufficiently small.
Also note that $| A (n , j, \bt ) | \leq C n^{-1/2} \| \bt \|$, so that
\begin{equation}
\label{A-cubed}
 | A (n,j,\bt) |^3 \leq C n^{-3/2} \| \bt \|^3 \leq C' \delta n^{-1} \| \bt \|^2 \leq \eps n^{-1} \| \bt \|^2 ,\end{equation}
for $\delta$ sufficiently small; here $C$ and $C'$ are constants that do not depend on $\delta$.
Thus we may apply~\eqref{eq:complex-taylor} to obtain
\begin{align*}
\prod_{j=1}^n \varphi ( n^{-3/2} j \bt ) & = \exp \left\{ \sum_{j=1}^n \log \left[ 1 + A (n,j ,\bt) + \Delta_1 (n,j,\bt) \right] \right\} \\
& = \exp \left\{ \sum_{j=1}^n \left( A(n,j,\bt) - \frac{1}{2} A(n,j,\bt)^2 \right) + \Delta_1 (n, \bt) \right\} ,\end{align*}
where $| \Delta_1 (n, \bt ) | \leq \eps \| \bt \|^2 $ for $\delta$ sufficiently small.
Here~\eqref{A-sq} holds, where now, for all $\bt \in  [0, \delta n^{1/2}]^d$, 
similarly to~\eqref{A-cubed}, $|\Delta_0 (n, j, \bt ) | \leq \eps n^{-1} \| \bt \|^2$
for $\delta$ sufficiently small (depending on $\eps$).
So, for $\delta$ sufficiently small, for $\bt \in R_2 (n)$,
\[ \prod_{j=1}^n \varphi ( n^{-3/2} j \bt )  = \exp \left\{  \frac{ i (n+1) \bt^\tra \mu}{2n^{1/2}}
- \frac{\bt^\tra M \bt}{6} + \Delta_2 ( n , \bt ) \right\} ,\]
where $| \Delta_2 (n , \bt ) | \leq \eps \| \bt \|^2$ for all $n$ sufficiently large. Then,
by~\eqref{eq:positive-definite} and choice of $\eps$,
$\bt^\tra M \bt \geq 12 \eps \| \bt \|^2$, so that, by choice of  $\delta$,
\begin{align*}
| \Phi_n (\bt ) |   = \left| \exp \left\{ - \frac{ \bt ^\tra M \bt}{6} + \Delta_2 ( n , \bt ) \right\} \right|  
  \leq  \exp \{ -   \eps \| \bt \|^2 \} , \text{ for all } \bt \in R_2 (n).
	\end{align*}
So we have 
\begin{align*} 
I_2 (n) = \int_{R_2 (n)} | \Phi_n ( \bt ) | \ud \bt 
 \leq \int_{\R^d \setminus R_1} \exp \left\{ - \eps \| \bt \|^2 \right\} \ud \bt,
\end{align*}
for $\delta$ sufficiently small and $n$ sufficiently large. This yields part (ii) of the lemma. 

Now we proceed to estimate $I_3(n)$. 
First note that, by~\eqref{eq:Phi-phi}, 
\begin{equation}
|\Phi_n(\bt)| = \prod_{j=1}^{n}|\varphi(n^{-3/2}j\bt)| \le \prod_{j=\lceil n/2 \rceil}^{n}|\varphi(n^{-3/2}j\bt)|. 
\label{eq:cal20}
\end{equation}
For any $\bt \in R_3(n)$, we have
 $n^{-3/2}j\bt \in  [-\pi j / n,  \pi j / n]^d \setminus   [-\delta j / n, \delta j / n]^d$. 
In particular
\[ \bigcup_{j=\lceil n/2 \rceil}^n \{ n^{-3/2}j\bt \} \subset  [-\pi , \pi  ]^d \setminus [-\delta  / 2, \delta / 2]^d . \]
Thus we may apply the final statement in
Lemma~\ref{lem:equivalent} for some $\rho$ sufficiently small 
to obtain
\[ \sup_{\bt \in R_3(n)} \sup_{\lceil n/2 \rceil \leq j \leq n} | \varphi ( n^{-3/2} j \bt ) | \leq \re^{-c_\rho} ,\]
for some $c_\rho >0$.
Hence from~\eqref{eq:cal20} we have  $\sup_{\bt \in R_3 (n)} |\Phi_n(\bt)| \le \re^{-n c_\rho / 2}$.
It follows that
\begin{align*}
I_3(n)
  \le \int_{[ - \pi \sqrt{n}, \pi \sqrt{n} ]^d} \re^{-n c_\rho / 2} 
  \leq (2\pi)^d n^{d/2} \re^{-n c_\rho / 2} .
\end{align*}
This gives part (iii) of the lemma. 

It remains to estimate $I_4(n)$. Fix $\bt \in R_4(n)$, and consider sets
\begin{align*}
\Lambda_n(\bt) = \left\{ n^{-3/2} j \bt : j \in \{1,2, \ldots, n\} \right\} , \text{ and } 
L_n(\bt) &= \left\{ n^{-3/2} u \bt : 1 \le u \le n \right\} .
\end{align*}
Recall that $S_H := 2\pi \Z^d$ in the case $H = I$, and, for $\rho >0$, define $S_H(\rho) := \cup_{\by \in S} B(\by;\rho)$, 
where $B(\by; \rho)$ is the open Euclidean ball of radius $\rho$ centred at $\by \in \R^d$.
Define $N_n(\bt) := | \Lambda_n(\bt) \setminus S_H(\rho)|$.
  Lemma~\ref{lem:equivalent} and~\eqref{eq:Phi-phi} show that
\begin{equation}
\label{eq:counting_bound}
|\Phi_n(\bt)| = \prod_{j=1}^{n}|\varphi(n^{-3/2}j\bt)| \le \exp \{ - c_\rho N_n(\bt)\},
\end{equation}
for some positive constant $c_\rho$. 
We aim to show that $N_n (\bt)$ is bounded below by a constant times $n$.
To do this we use a counting argument related  to one used in~\cite[Lemma~4.4]{DH}.

Let $K_n(\bt)$ be the number of $\bx \in S_H$ such that $B(\bx ; \rho) \cap L_n(\bt) \neq \emptyset$. 
Set
$\kappa := n^{-3/2}\|\bt\|$. As $\bt \in [-\pi n^{3/2}, \pi n^{3/2}]^d \setminus [-\pi n^{1/2}, \pi n^{1/2}]^d$, we have
\begin{equation}
\label{nu-bounds}
\frac{\pi}{n} \le \kappa \le \pi \sqrt{d}.
\end{equation}
Take $\rho = \pi / 8$. We claim that between any two balls of $S_H(\rho)$ that intersect $L_n(\bt)$
there is at least one point of $\Lambda_n (\bt) \setminus S_H (\rho)$. 
Write $\by_j = n^{-3/2} j \bt$ for $j \in \{1,\ldots, n\}$. Suppose $i_1, i_2 \in \{1,\ldots, n\}$
with $i_1 < i_2$ and $\bx_1, \bx_2 \in S_H$ with $\bx_1 \neq \bx_2$ are such that $\by_{i_1} \in B (\bx_1 ; \rho)$
and $\by_{i_2} \in B (\bx_2, \rho)$. To prove the claim we need to show that there exists $j$ with $i_1 < j < i_2$
such that $\by_j \notin S_H(\rho)$. First note that since $n^{-3/2} \bt \in [ -\pi, \pi]^d$ and 
$\by_{i_1} \in B (\bx_1 ; \rho)$, the point $\by_{i_1 +1}$ must lie in the box $Q(\bx_1) = \bx_1 + [-9\pi/8,9\pi/8]^d$.
As $9\pi/8 < 15\pi/8 = 2\pi - \rho$, the box $Q(\bx_1)$ does not intersect any balls in $S_H(\rho)$ other than $B (\bx_1 ; \rho)$.
There are two cases. Either (i) $\by_{i_1 +1} \notin B (\bx_1 ; \rho)$, or (ii)
$\by_{i_1 +1} \in B (\bx_1 ; \rho)$. In case (i) the claim is proved. In case (ii), we have $\kappa \leq 2 \rho$, and since
$B (\bx_1 ; 3 \rho)$ is contained in $Q(\bx_1)$, there is some $j$ with $i_1 + 1 < j < i_2$ such that $\by_j \notin S_H(\rho)$,
proving the claim.
Hence
\begin{equation}
\label{N-lower-bound2}
 N_n (\bt) \geq   K_n (\bt) - 1  
 .\end{equation}
The total length of the line $L_n(\bt)$ is less than $\kappa n$,
and each segment of $L_n (\bt)$ between neighbouring balls that intersect $L_n (\bt)$
has length at least $2\pi - 2\rho$, so 
$(K_n (\bt) -1 ) (2\pi - 2\rho) \leq \kappa n$, or, equivalently,
\begin{equation}
\label{K-upper-bound}
K_n(\bt) \le \frac{4\kappa n}{7\pi} +1.
\end{equation}
Moreover, each ball of $S_H(\rho)$ that intersects $L_n (\bt)$ contains
at most $1 + 2 \rho/ \kappa$ points of $\Lambda_n (\bt)$, so that 
the number of points in $\Lambda_n(\bt) \cap S_H(\rho)$ satisfies
\begin{equation}
\label{N-lower-bound3}
n-N_n(\bt) \leq K_n(\bt) \left(\frac{\pi}{4\kappa}+1\right).
\end{equation}
Let $\eps >0$ be a constant. We consider the following two cases. \\
\emph{Case 1:} $K_n(\bt) \le \eps n \kappa$.
In this case we have from~\eqref{N-lower-bound3} and~\eqref{nu-bounds} that
\begin{equation*}
N_n(\bt) \ge n- \frac{\pi \eps}{4}n - \eps n\kappa \ge n- \frac{\pi}{4}\eps n - \eps n \pi \sqrt{d}  \ge \eps n,
\end{equation*}
for $\eps$ small enough. \\
\emph{Case 2:} $K_n(\mathbf{t}) > \eps n \kappa$.
If $\kappa \ge \frac{1}{2}$, then we have from~\eqref{N-lower-bound2} that,
\begin{equation*}
N_n(\mathbf{t}) \ge K_n(\mathbf{t})-1 \ge (\eps/3) n ,
\end{equation*}
for $n$ sufficiently large.
On the other hand, if $\kappa < \frac{1}{2}$, then~\eqref{N-lower-bound3} and~\eqref{K-upper-bound} show 
that
\begin{align*}
N_n(\bt) & \ge  n- \left(\frac{4\kappa n}{7 \pi } + 1\right)\left(\frac{\pi}{4\kappa} +1\right) \\
& = \frac{6n}{7} - \frac{\pi}{4 \kappa} - \frac{4 \kappa n}{7 \pi} -1 \\
& \geq \frac{6n}{7} - \frac{n}{4} - \frac{2 n}{7 \pi} -1 ,
\end{align*}
by~\eqref{nu-bounds} and the assumption $\kappa < \frac{1}{2}$.
Thus we have shown that, in any case,
$N_n(\bt) \ge \eps n$
for some constant $\eps >0$ and all $n$ sufficiently large. Thus from~\eqref{eq:counting_bound} we conclude that 
\begin{align*}
  I_4 (n)    = \int_{R_4(n)} | \Phi_n (\bt) | \ud \bt  
  \le ( 2\pi n^{3/2} )^d  \exp\left\{-\eps c_\rho n \right\}  .
\end{align*}
Hence we have proved the last statement in Lemma~\ref{lem:parts1}.
\end{proof}

Now we can gather all our estimates and complete the proof of Theorem~\ref{thm:LCLT}. 

\begin{proof}[Proof of Theorem~\ref{thm:LCLT}]
As explained at the start of this section,
it suffices to prove the case where $\bb = \0$ and $H = I$ (so $h=1$). Then we have from Lemma~\ref{lem:lclt-first} that
\begin{align*}
\sup_{\bx \in \cL_n} \left| n^{3d/2}  p_n(\bx) - \nu \left(\bx - \frac{(n+1)}{2n^{1/2}} \bmu \right) \right| 
&\le \sum_{k=1}^{4} I_k(n) + \int_{\R^d \setminus R_1} \exp \left\{ -\frac{\lambda_{\rm min} (M)}{6} \| \bt \|^2 \right\} \ud \bt.
\end{align*}
Fix $\eps >0$. Then we can choose $\ell$ sufficiently large such that the integral in the above display is less than $\eps$,
and, by Lemma~\ref{lem:parts1}, also $I_2(n) \le \eps$ for all $n$; fix such an $\ell$.
Then Lemma~\ref{lem:parts1}  shows that
$I_1 (n) + I_3 (n) + I_4 (n) \to 0$ as $n \to \infty$. Since $\eps >0$ was arbitrary, the proof of the theorem is completed.
\end{proof}

\section{One dimension}
\label{sec:one-dimension}

We start with a couple of general observations. Recall that an event defined
in terms of a sequence of random variables $X_1, X_2, \ldots$
is \emph{permutable}
if its occurrence is a.s.~invariant under any finite permutation of $X_1, X_2, \ldots$:
see~\cite[p.~232]{ct} for a formal definition.

\begin{lemma}
\label{lem:exchangeable}
Let $d=1$.
For any $x \in \R$, the event $\left\{ \limsup_{n \to \infty} G_n \ge x \right\}$ is permutable.
\end{lemma}
\begin{proof}
For any $x \in \R$, we notice that for any positive integer $k$,
\begin{align}
\label{eq:exc}
\left\{ \limsup_{n \to \infty} G_n \ge x \right\} &= \left\{ \limsup_{n \to \infty} \left[ \frac{1}{n}(S_1 +S_2 + \cdots + S_k) + \frac{1}{n}(S_{k+1} +S_{k+2} + \cdots + S_n) \right] \ge x \right\} \nonumber\\
&= \left\{ \limsup_{n \to \infty}  \frac{1}{n}(S_{k+1} +S_{k+2} + \cdots + S_n) \ge x \right\},
\end{align}
up to events of probability $0$, since $\lim_{n \to \infty} \frac{1}{n} (S_1 + \cdots + S_k ) = 0$, a.s.
But the event on the right-hand side of~\eqref{eq:exc} is invariant under permutations of $X_1, X_2, \ldots, X_k$.
\end{proof}

\begin{lemma}
\label{lem:4case}
Let $d=1$. One and only one of the following will occur with probability $1$.
\begin{itemize}
\item[(i)] $G_n=0$ for all $n$.
\item[(ii)] $G_n \to \infty$.
\item[(iii)] $G_n \to -\infty$.
\item[(iv)] $-\infty = \liminf_{n\to\infty} G_n < \limsup_{n\to\infty} G_n = \infty$.
\end{itemize} 
\end{lemma}
\begin{proof}
We adapt the proof of Theorem~4.1.2 in \cite{RD}.
Lemma~\ref{lem:exchangeable} and the Hewitt--Savage zero--one law (see e.g.~\cite[p.~238]{ct}) imply $\limsup_{n\to\infty} G_n = \lambda$, a.s.,
for some  $\lambda \in [-\infty, \infty]$.
Let $G'_n := \frac{n+1}{n}(G_{n+1}-X_1) = \sum_{i=1}^{n}\frac{n-i+1}{n}X_{i+1}$. Recalling~\eqref{eq:weighted-sum}, we see the sequence $(G'_n)$ has the same distribution as $(G_n)$. 
So taking $n \to \infty$ in $\frac{n}{n+1} G'_n = G_{n+1}-X_1$ we obtain
$\lambda=\lambda-X_1$, a.s., implying $X_1=0$ a.s.~if $\lambda$ is finite, which is case (i). 
Otherwise, $\lambda = -\infty$ or $+\infty$. A similar argument applies to $\liminf_{n\to\infty} G_n$. 
The 3 possible combinations ($\limsup_{n\to \infty} G_n = -\infty$ and $\liminf_{n\to\infty} G_n = \infty$ being impossible) give (ii), (iii), and (iv).
\end{proof}

Clearly cases (ii)~and (iii) of Lemma~\ref{lem:4case} are transient; case~(iv),
when the walk oscillates, is the most interesting case. The next result shows that
oscillating behaviour is enough to ensure recurrence provided that $\E X =0$.

\begin{lemma}
\label{lem:change_sign}
Suppose that $d=1$ and $\E  X =0$. Suppose that 
$\limsup_{n \to \infty} G_{n}= + \infty$ and $\liminf_{n \to \infty} G_{n}= - \infty$.
Then, for any $x \in \R$, $\liminf_{n \to \infty}|G_n - x|=0$, a.s.
\end{lemma}
\begin{proof}
Fix $\eps >0$. 
Since $S_n /n \to 0$ a.s.~and, by Proposition~\ref{prop:LLN}, $G_n /n \to 0$ a.s., we have
\[
G_{n+1}-G_{n} = \frac{S_{n+1}-G_n}{n+1} \to 0 , \as 
\]
Hence $| G_{n+1} - G_n | < \eps$ for all but finitely many $n$.
For any $x \in \R$, $\limsup_{n \to \infty} G_{n}= + \infty$ and $\liminf_{n \to \infty} G_{n}= - \infty$
implies that there are infinitely many $n$ for which $G_n - x$ and $G_{n+1} -x$ have opposite signs.
Hence $| G_n - x | < \eps$ i.o.
\end{proof}

The next result shows that $G_n$ does oscillate when~\eqref{ass:basicd} holds.  

\begin{lemma}
\label{lem:change_sign2}
Suppose that $d=1$, that~$\E [ X^2 ] \in (0,\infty)$, and that $\E X = 0$.
 Then $\limsup_{n \to \infty} G_{n}= + \infty$ and $\liminf_{n \to \infty} G_{n}= - \infty$.
\end{lemma}
\begin{proof}
For any $x \in \R$, we have that
\begin{align*}
\Pr \left(\limsup_{n \to \infty} G_n \ge x \right) &\ge \Pr \left( G_n \ge x \text{ i.o.} \right) \\
&= \Pr \left( \bigcap_{m=1}^\infty \bigcup_{n \ge m} \{G_n \ge x\} \right) \\
&= \lim_{m \to \infty} \Pr \left( \bigcup_{n \ge m} \{G_n \ge x\} \right) \\
&\ge  \lim_{m \to \infty} \Pr \left( G_m \ge x \right) \\
&= \frac{1}{2},
\end{align*}
by the central limit theorem, Proposition~\ref{prop:CLT}. 
With Lemma~\ref{lem:exchangeable} and the Hewitt--Savage zero--one law (see e.g.~\cite[p.~238]{ct}), it follows that
 $\limsup_{n \to \infty} G_n \ge x$, a.s., 
and since $x \in \R$ was arbitrary, we get
 $\limsup_{n \to \infty} G_n= +\infty$. A similar argument gives 
$\liminf_{n \to \infty} G_n= -\infty$.  
\end{proof}
 
\begin{proof}[Proof of Theorem~\ref{thm:classification}.]
Under the conditions in part~(i) of the theorem, the process $(G_n)$ has the same distribution as the process $(-G_n)$,
and so we must be in either case (i) or (iv) of 
 Lemma~\ref{lem:4case}. The trivial case (i) is ruled out since $\E | X | >0$.  Thus case (iv) applies, and
 $G_n$ changes sign i.o., so by Lemma~\ref{lem:change_sign} we obtain the desired conclusion.

Under the conditions in part~(ii), Lemma~\ref{lem:change_sign2} applies, so  Lemma~\ref{lem:4case}(iv) applies again, and the same argument
gives the result.
\end{proof}

For the remainder of this section we work towards a proof of Theorem~\ref{thm:stable-transience}.
 The proof rests on the following local limit theorem. We use the notation
\[ \cL_n := \left\{ n^{-1-1/\alpha} \left( \tfrac{1}{2} n (n+1) b + h \Z \right) \right\} ,\]
and $p_n (x) := \Pr ( G_n = n^{1/\alpha} x )$.

\begin{theorem}
\label{thm:SLCLT}
Suppose that~$d=1$ and~\eqref{ass:basicd2} holds, i.e., $\Pr(X \in b + h \Z)=1$ for $b \in \R$ and $h >0$ maximal. 
Suppose also that~\eqref{ass:stable} holds.  Then  
\begin{equation}
\lim_{n \to \infty} \sup_{x \in \cL_n} \left| \frac{n^{1+1/\alpha}}{h} p_n(x)-(\alpha+1)^{1/\alpha} g\left((\alpha+1)^{1/\alpha}x\right) \right| = 0, 
\label{eq:slclt}
\end{equation}
where $g(x)$ is the density of the stable distribution in~\eqref{ass:stable}.
\end{theorem}
\begin{proof} 
The proof is similar to that of~Theorem~\ref{thm:LCLT}, and can also be compared 
to the proof of the local limit theorem for sums of i.i.d.~random variables
in the domain of attraction of a stable law: see \cite[\S 4.2]{IL}.

Assumption~\eqref{ass:stable} implies that $n^{-1/\alpha} S_n$ converges
in distribution to a (constant multiple of) a random variable with characteristic function
$s (t) = \re^{-c |t|^\alpha}$, where $c >0$ and $\alpha \in (0,1)$; see Theorems~2.2.2 and~2.6.7 of~\cite{IL}. It also follows, by an examination of the statements of Theorems~2.6.1 and~2.6.7 of~\cite{IL} and the proof of Theorem~2.6.5 of~\cite{IL}, that for $t$ in a neighbourhood of $0$,
\begin{equation}
\label{domain}
 \log \varphi (t) = - c | t |^\alpha \left( 1 + \eps (t) \right) ,\end{equation}
where $| \eps (t) | \to 0$ as $t \to 0$.
 
Define $Y_n = \sum_{i=1}^n S_i$ and let
\[ \Phi_n(t) := \E \re^{in^{-1-1/\alpha}t Y_n}. \]
Using the $d=1$ case of the inversion formula~\eqref{eq:formula} with $W=\left( Y_n -\frac{n(n+1)}{2}b \right)/h \in \Z$, we get  
\begin{align*}
p_n(x )  = 
\frac{1}{2\pi } 
\int_{-\pi}^\pi \re^{ -\frac{iu}{h} \left( n^{1+1/\alpha} x -\frac{n(n+1)}{2}b \right)  } \E\left[ \re^{\frac{iu}{h} \left( Y_n -\frac{n(n+1)}{2}b \right)  }\right] \ud u, \text{ for } x \in \cL_n.
\end{align*}
Using the substitution $t = u n^{1+1/\alpha}/h$, we obtain
\begin{equation}
\frac{n^{1+1/\alpha}}{h} p_n(x ) = \frac{1}{2\pi} \int_{-\pi n^{1+1/\alpha}/h}^{\pi n^{1+1/\alpha}/h} \re^{-itx} \Phi_n(t)\ud t. 
\label{cal51}
\end{equation}
On the other hand, 
from the inversion formula for densities we have that
\begin{equation}
\label{eq:stable-g}
g (x )= \frac{1}{2\pi} \int_{-\infty}^{\infty} \re^{-itx}s (t )\ud t,
\end{equation}
where $g$ is the density corresponding to $s$. It follows that
\begin{align*}
(\alpha+1)^{1/\alpha} g\left((\alpha+1)^{1/\alpha}x\right) &=\frac{1}{2\pi} \int_{-\infty}^{\infty} (\alpha+1)^{1/\alpha} \re^{-it(\alpha+1)^{1/\alpha} x}s\left(t\right)\ud t \\
&=\frac{1}{2\pi} \int_{-\infty}^{\infty} \re^{-iux}s\left(\frac{u}{(\alpha+1)^{1/\alpha}}\right)\ud u,
\end{align*}
using the substitution $u = (\alpha+1)^{1/\alpha} t$.
Since $s(t)=\re^{-c|t|^\alpha}$, we get
\begin{equation}
(\alpha+1)^{1/\alpha} g\left((\alpha+1)^{1/\alpha}x\right)=\frac{1}{2\pi} \int_{-\infty}^{\infty} \re^{-itx - \frac{c|t|^\alpha}{\alpha+1}}\ud t.
\label{cal52}
\end{equation}
Subtracting equation~\eqref{cal52} from equation~\eqref{cal51}  we obtain
\begin{equation*}
\sup_{x \in \cL_n}
\left| \frac{n^{1+1/\alpha}}{h} p_n(x ) - (\alpha+1)^{1/\alpha} g\left((\alpha+1)^{1/\alpha}x\right) \right| \le \sum_{k=1}^4 J_k (n) + J_5,
\end{equation*}
where 
\begin{align*}
J_1(n) &:= \int_{-\ell}^\ell \left| \Phi_n(t) -  \re^{- \frac{c|t|^\alpha}{\alpha+1}} \right| \ud t \\
J_2(n) &:= \int_{\ell \le |t| \le \delta n^{1/\alpha}} \left| \Phi_n(t) \right| \ud t \\
J_3(n) &:= \int_{\delta n^{1/\alpha} \le |t| \le \pi n^{1/\alpha}/h} \left| \Phi_n(t) \right| \ud t \\
J_4(n) &:= \int_{\pi n^{1/\alpha}/h \le |t| \le \pi n^{1+1/\alpha}/h} \left| \Phi_n(t) \right| \ud t \\
J_5 &:= \int_{|t| > \ell} \left| \re^{- \frac{c|t|^\alpha}{\alpha+1}} \right| \ud t \\
\end{align*}
for some constants $\ell$ and $\delta$ to be determined later. 
The statement of the theorem will follow once we show that 
\[ \lim_{\ell \to \infty} \limsup_{n \to \infty} \left( \sum_{k=1}^4 J_k (n) + J_5 \right) = 0 .\]
Thus it remains to establish this fact.

Since $Y_n$ has the same distribution as $\sum_{j=1}^n jX_j$, we get
\begin{align}
\log \Phi_n(t) & = \log \prod_{j=1}^{n} \varphi\left( \frac{jt}{n^{1+1/\alpha}} \right) 
= \sum_{j=1}^n \log \varphi\left( \frac{jt}{n^{1+1/\alpha}} \right) \nonumber\\
& = - \frac{c | t|^\alpha}{n^{\alpha+1}} \sum_{j=1}^n j^\alpha  \left( 1 + \eps\left(\frac{jt}{n^{1+1/\alpha}} \right) \right),
\label{cal54}
\end{align}
using~\eqref{domain}. Since $|\eps(t)|\to0$ as $t \to 0$, we have
\begin{equation}
\label{epsilon-bound}
\lim_{n \to \infty} \sup_{t \in [-\ell, \ell]} \max_{j \in \{1,2, \ldots , n\}} \left| \eps\left(\frac{jt}{n^{1+1/\alpha}} \right) \right| =0.
\end{equation}
A simple consequence of the fact that $\sum_{k=0}^{n-1} k^\alpha \le \int_0^n u^\alpha \ud u  \le \sum_{k=1}^{n} k^\alpha$ for $\alpha > 0$ is
\begin{equation}
\label{sum-alpha}
\sum_{j=1}^n j^\alpha = \frac{n^{\alpha+1}}{\alpha+1} + O (n^\alpha ) .
\end{equation}
It follows from~\eqref{cal54}, \eqref{epsilon-bound}
and~\eqref{sum-alpha} that
\begin{equation}
\label{eq:66}
 \log \Phi_n (t) =  - \frac{c |t |^\alpha}{\alpha+1}  + \Delta (n,t)   ,\end{equation}
where $\sup_{t \in [-\ell,\ell]} | \Delta (n,t) | \to 0$ as $n \to \infty$.
It follows that $\lim_{n \to \infty} J_1 (n)=0$ for any $\ell \in (0,\infty)$.

For $J_2(n)$,  we see that 
\begin{equation*}
\lim_{\delta \to 0} \sup_n \sup_{t \in [-\delta n^{1/\alpha}, \delta n^{1/\alpha}]} \max_{j \in \{1,2, \ldots , n\}} \left| \eps\left(\frac{jt}{n^{1+1/\alpha}} \right) \right| =0.
\end{equation*}
So by~\eqref{cal54} and~\eqref{sum-alpha} we have that~\eqref{eq:66}
holds for $t \in [-\delta n^{1/\alpha}, \delta n^{1/\alpha}]$ where, 
choosing $\delta$ sufficiently small, we have that for all $n$ sufficiently large and
all $t \in [-\delta n^{1/\alpha}, \delta n^{1/\alpha}]$,
$| \Delta (n,t) | \leq \frac{1}{2} \frac{c |t|^\alpha}{\alpha+1}$.
Hence for sufficiently large $n$, for all $t \in [-\delta n^{1/\alpha}, \delta n^{1/\alpha}]$,
\begin{equation*}
| \Phi_n(t) | \le \exp \left\{ - \frac{1}{2} \frac{c |t|^\alpha}{\alpha+1} \right\} .
\end{equation*}
It follows that, for all $n$ sufficiently large,
\begin{equation*}
\sup_n J_2(n) \le  \int_{ |t| \ge \ell} \re^{-\frac{1}{2} \frac{c |t|^\alpha}{\alpha+1}} \ud t ,
\end{equation*}
which tends to $0$ as $\ell \to \infty$. 

Next we consider $J_3(n)$. First observe that
\begin{equation*}
\left| \Phi_n(t) \right| = \prod_{j=1}^{n} \left| \varphi\left( \frac{jt}{n^{1+1/\alpha}} \right) \right| \le \prod_{j=\lceil n/2 \rceil }^{n} \left| \varphi\left( \frac{jt}{n^{1+1/\alpha}} \right) \right|.
\end{equation*}
Now for any $\delta n^{1/\alpha} \le |t| \le \pi n^{1/\alpha}/h $ and any $\lceil n/2 \rceil \le j \le n$, we have 
\begin{equation*}
\frac{\delta}{2} \le \left| \frac{jt}{n^{1+1/\alpha}} \right| \le  \frac{\pi}{h}.
\end{equation*}
We can take  $\rho$ sufficiently small so that 
\begin{equation*}
\rho < \frac{\delta}{2} \le \left| \frac{jt}{n^{1+1/\alpha}} \right| \le  \frac{\pi}{h} < \frac{2\pi}{h}- \rho.
\end{equation*}
So an application of the $d=1$ case of Lemma~\ref{lem:equivalent} gives, for all $n$,
\begin{equation*}
\sup_{\delta n^{1/\alpha} \le |t| \le \pi n^{1/\alpha}/h} \sup_{\lceil n/2 \rceil \le j \le n} \left| \varphi\left( \frac{jt}{n^{1+1/\alpha}} \right) \right| \le \re^{-c_\rho},
\end{equation*}
for some $c_\rho >0$. Hence we have  
\begin{equation*}
\sup_{\delta n^{1/\alpha} \le |t| \le \pi n^{1/\alpha}/h} |\Phi_n(t)| \le \re^{-n c_\rho / 2},
\end{equation*}
and hence
\begin{equation*}
J_3(n) = \int_{\delta n^{1/\alpha} \le |t| \le \pi n^{1/\alpha}/h} \left| \Phi_n(t) \right| \ud t \le
\frac{2\pi}{h} n^{1/\alpha} \re^{-n c_\rho / 2} \to 0,
\end{equation*}
as $n \to \infty$.

For $J_4(n)$, we follow essentially the same counting argument as that used for $I_4(n)$ in Section~\ref{sec:LCLT-proof}.
Let $t' = t/h$. 
Define
\begin{align*}
\Lambda' (t') := \left\{ n^{-1-1/\alpha}jt' : j \in \{1,2, \ldots, n\} \right\} \text{ and } L'_n(t') := \left\{ n^{-1-1/\alpha}ut' : 1 \le u \le n \right\}
\end{align*}
Let $\kappa := n^{-1-1/\alpha} | t '|$ denote the spacing of the points of $\Lambda' (t')$.
Set $N_n (t') := | \Lambda' (t') \setminus S (\rho) |$, where $S (\rho ) := \cup_{x \in 2\pi \Z} (x -\rho, x+\rho)$.
Since $\pi n^{1/\alpha} \le |t'| \le \pi n^{1+1/\alpha}$, we have
$\frac{\pi}{n} \le \kappa \le \pi$,
which is just the $d=1$ case of~\eqref{nu-bounds}. The counting argument in Section~\ref{sec:LCLT-proof} is based
on the fact that there are $n$ points with spacing satisfying~\eqref{nu-bounds}, so the
 argument goes through unchanged to give $N_n (t') \geq \eps n$, and we get
\[
J_4(n) = \int_{\pi n^{1/\alpha}/h \le |t| \le \pi n^{1+1/\alpha}/h} \left| \Phi_n(t) \right| \ud t \le \frac{2\pi}{h} n^{1+1/\alpha}  \exp\left\{- \eps c_\rho n \right\} \to 0,
\]
as $n \to 0$.

Finally, it is clear that $\lim_{\ell \to \infty} \sup_n J_5 = 0$.
\end{proof}

\begin{proof}[Proof of Theorem~\ref{thm:stable-transience}]
First note that the assumption~\eqref{ass:stable} implies that we are in case (iv) of Lemma~\ref{lem:4case},
so that $\liminf_{n \to \infty} G_n= -\infty$ and $\limsup_{n \to \infty} G_n= +\infty$.

It remains to prove that $|G_n | \to \infty$.
Fix $x \in (0,\infty)$ and consider the interval $I = (-x,x)$.
Then $\Pr ( G_n \in I) = \Pr (n^{-1/\alpha} G_n \in n^{-1/\alpha} I )$.
Since the lattice spacing of $\mathcal{L}_n$ is
of order $n^{-1-1/\alpha}$,
 the interval $n^{-1/\alpha} I$ contains $O ( n)$ lattice points of $\cL_n$.
Theorem~\ref{thm:SLCLT} and the fact that, by~\eqref{eq:stable-g}, $\sup_x g(x) < \infty$,
shows that each such lattice point is associated with probability
$O ( n^{-1-1/\alpha})$. So we get $\Pr (G_n \in I ) = O (n^{-1/\alpha})$,
which is summable for $\alpha \in (0,1)$. Hence the Borel--Cantelli lemma implies that
$\liminf_{n \to \infty} | G_n| \geq x$, a.s., and since $x$ was arbitrary the result follows.
\end{proof}

\section{Transience and rate of escape}
\label{sec:transience}
 
This section is devoted to the proof of Theorem~\ref{thm:classification2} for $d \geq 2$. 
The idea is to use the local limit theorem to control (via Borel--Cantelli) the visits of $G_n$
to a growing ball, along a subsequence of times suitably chosen so that
the slow movement of the centre of mass controls the trajectory between the times
of the subsequence as well. Here is our estimate on the deviations, which is valid for any $d \in \N$.

\begin{lemma}
\label{lem:diff}
Suppose that~\eqref{ass:basicd} holds and that $\bmu = \0$. 
Let $a_n=\lceil n^\beta \rceil$  for some $\beta >1$. Then, for any $\varepsilon>0$, a.s. for all but finitely many $n$,  
\[
\max_{a_n \le m \le a_{n+1}} \|G_m-G_{a_n} \| \le n^{\frac{\beta}{2}-1+\varepsilon}.
\]
\end{lemma}
\begin{proof}
We use the crude bound that for any $\eps>0$,
 $\| S_n \| \leq n^{(1/2)+\eps}$ all but finitely often (f.o.), a.s. From this and the triangle inequality, it follows that
\begin{equation}
\label{upperbound}
\|G_n\| \le \frac{1}{n}  \sum_{i=1}^n \| S_i  \| \le \max_{1 \le i \le n} \|S_i\| \le n^{(1/2)+\eps},
\end{equation}
all but f.o., a.s.
Next, by the triangle inequality again, for any $\eps>0$, a.s., all but f.o.,  
\begin{equation}
\|G_{n+1}-G_n\| = \left\| \frac{S_{n+1}-G_n}{n+1} \right\| \le \frac{\|S_{n+1}\|}{n+1} + \frac{\|G_n\|}{n+1} \le n^{-(1/2)+\eps} .
\label{eq:cal7}
\end{equation}
It follows that for any $\eps>0$, a.s., all but f.o., 
\begin{align*}
\max_{a_n \le m \le a_{n+1}} \|G_m-G_{a_n} \| & = \max_{a_n \leq m \leq a_{n+1}} \left\| \sum_{j=a_n}^{m-1} (G_{j+1} - G_j ) \right\| \\
& \le \left( a_{n+1} - a_n \right) \max_{a_n \le m \le a_{n+1}-1} \|G_{m+1}-G_m\| , \end{align*}
where $a_{n+1} -a_n   \leq (n+1)^\beta -n^\beta +1 = O (n^{\beta -1} )$, and, a.s., all but f.o., by~\eqref{eq:cal7},
\begin{align*}  
\max_{a_n \le m \le a_{n+1}-1} \|G_{m+1}-G_m\| & \leq a_n^{-(1/2)+\eps} = O ( n^{-(\beta/2) + \beta \eps} ) .\end{align*}
 Since $\eps>0$ was arbitrary, the result follows.
\end{proof}
 
Now we are ready to prove Theorem~\ref{thm:classification2}.

\begin{proof}[Proof of Theorem~\ref{thm:classification2}.]
First, given the upper bound in equation~\eqref{upperbound}, we only need to show that for any $\eps>0$, a.s., for all but finitely many $n$, 
\begin{equation}
\label{lowerbound}
\|G_n\| \ge n^{(1/2)-\eps}.
\end{equation}
Let $B(r)$ denote the closed Euclidean ball, centred at the origin, of radius $r>0$.
We show that for any $\gamma \in (0, 1/2)$, $G_n$ will return to the ball $B(n^\gamma)$ only f.o. 
To do this, we show that along a suitable subsequence $a_n = \lceil n^\beta \rceil$, $\beta >1$,
$G_{a_n}$ returns to the ball
$B(2a_n^\gamma)$ only f.o., and Lemma~\ref{lem:diff} controls the trajectory between the instants of the subsequence.

First, we claim that
\begin{equation}
\Pr(G_n \in B(2n^\gamma) ) \le Cn^{d\left(\gamma- \frac{1}{2}\right)}, 
\label{eq:correctorder}
\end{equation}
for sufficiently large $n$ and some constant $C$.
Then 
\[
\sum_{n=1}^{\infty} \Pr(G_{a_n} \in B(2a_n^\gamma) ) \le C \sum_{n=1}^{\infty} n^{\beta d\left(\gamma- \frac{1}{2}\right)}. 
\]
Assuming that 
\begin{equation}
\beta > \frac{2}{d(1-2\gamma)} \label{betacon1}
\end{equation}
this sum converges, so the Borel--Cantelli lemma shows that $G_{a_n} \notin B(2a_n^\gamma)$ for all but finitely many $n$, a.s.
It then follows from Lemma~\ref{lem:diff} that between any $a_n$ and $a_{n+1}$ with $n$ sufficiently large,
the trajectory deviates by at most $n^{(\beta/2)-1+\eps}$. In particular, the trajectory between times $a_n$ and $a_{n+1}$ will not visit $B(a_n^\gamma)$ if we ensure that
$n^{(\beta/2)-1+\eps} < a_n^\gamma$. (See Figure~\ref{fig1}.) The latter condition can be achieved (for sufficiently small choice of $\eps$) if $(\beta/2) -1 < \beta \gamma$,
i.e., $\beta <  (\frac{1}{2} -\gamma )^{-1}$. Combined with~\eqref{betacon1} we see that we must choose $\beta$ such that
\[
\max \left\{ 1, \frac{2}{d(1-2\gamma)} \right\} < \beta < \frac{2}{(1-2\gamma)},
\]
which is possible for any $\gamma \in (0,1/2)$, provided $d \geq 2$.

Consider $n$ such that $a_m \leq n < a_{m+1}$; then we have shown that a.s., for all but finitely many $n$,
\[ \| G_n \| \geq a_m^\gamma  \geq m^{\beta \gamma} 
\geq \left( \frac{m^{\beta \gamma} }{2 (m+1)^{\beta \gamma} } \right) a^\gamma_{m+1}   .\]
In particular, for all $n$ sufficiently large, $\| G_n \| \geq (1/4) n^\gamma$, 
which establishes~\eqref{lowerbound}.

\begin{figure}[!h]
\center
\includegraphics[width=60mm]{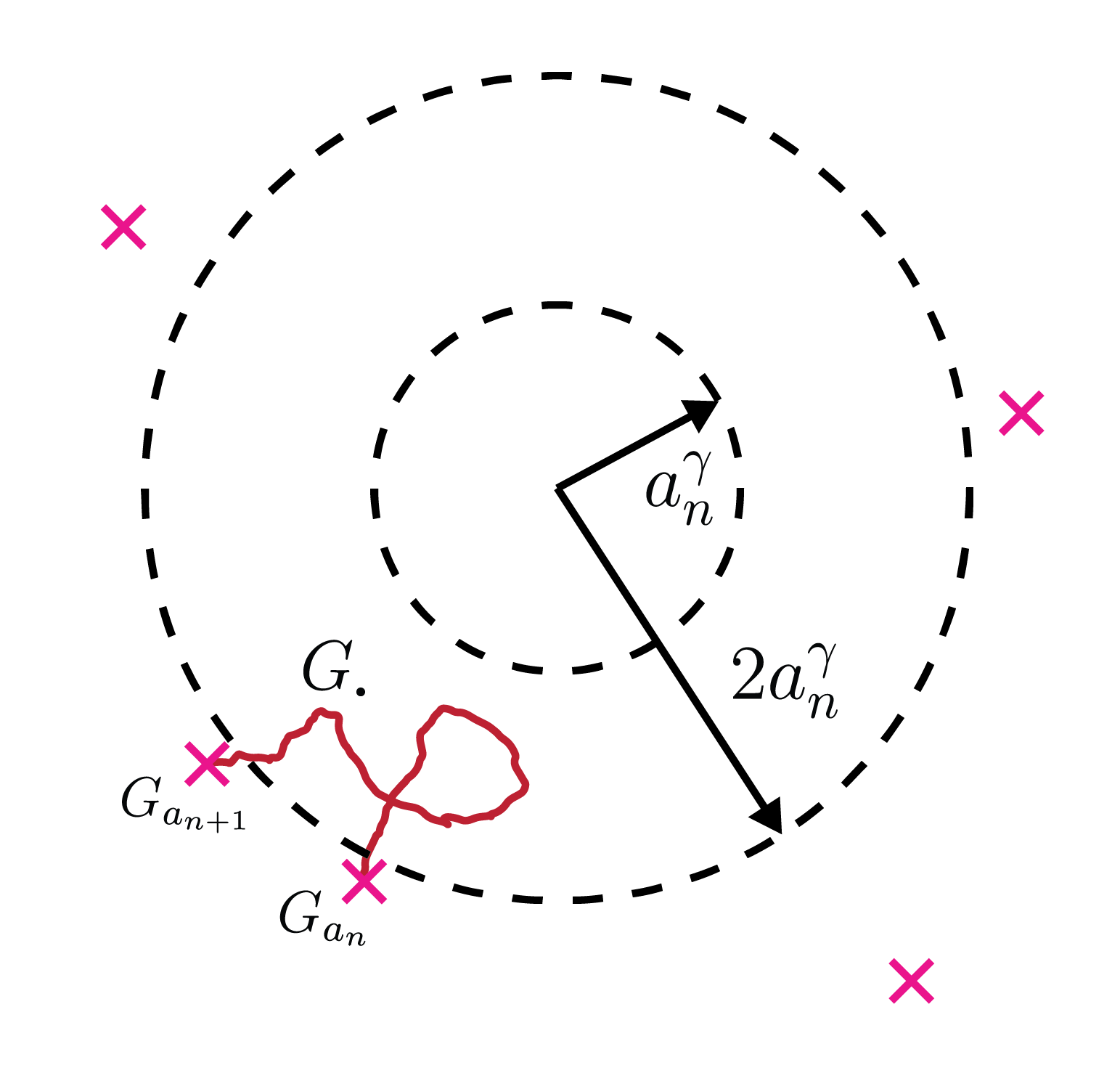}
\caption{Controlling $G_n$ along a subsequence.}
\label{fig1}
\end{figure}

It remains to prove the claim~\eqref{eq:correctorder}; here we use our local limit theorem.
First note that
\[
\Pr(G_{n} \in B(2n^\gamma) )=\Pr(n^{-1/2}G_{n} \in n^{-1/2} B(2n^\gamma) ).
\]
Since $| \det H | \in (0,\infty)$, the set $n H^{-1} B(2n^\gamma)$ 
is contained in a hypercube of side length $O(n^{\gamma+1})$,
and this hypercube contains $O ( n^{d(\gamma+1)} )$ points of any
translation of $\Z^d$. Hence 
$n^{-1/2} B(2n^\gamma)$  contains $O(n^{d(\gamma+1)})$ lattice points of $\cL_n$. 
From Theorem~\ref{thm:LCLT}, we also know that  for all $\bx \in \mathcal{L}_n$, 
$
\Pr(n^{-1/2}G_n = \bx) = O (n^{-3d/2} )
$.
Summing up over all $\bx \in n^{-1/2} B(2n^\gamma)$ we get
\[
\Pr(n^{-1/2} G_{n} \in n^{-1/2}B(2n^\gamma)) = O \left(n^{-3d/2} \times n^{d(\gamma+1)} \right)= O\left(n^{d\left(\gamma- \frac{1}{2}\right)}\right),
\]
establishing~\eqref{eq:correctorder}. This completes the proof.
\end{proof}

\appendix
\section{Lattice distributions and characteristic functions}
\label{sec:char-fn}

Recall that $\varphi (\bt) := \E \re^{i \bt^\tra X}$ is the ch.f.~of $X$.

\begin{lemma}
\label{lem:estchf}
Suppose that $\E [ \| X \|^2 ] < \infty$.
For any $\bt \in \R^d$,
\begin{equation}
\varphi(\bt)=1+i\bt^\tra \E X -\frac{1}{2}\bt^\tra \E[X X^\tra]\bt + \|\bt\|^2 W(\bt), 
\label{eq:estchf}
\end{equation}
where for any $\eps >0$, there exists $\delta >0$ such that $|W(\bt)| \le \eps$ for all $\bt$ with $\|\bt\| \le \delta$.
\end{lemma}
\begin{proof}
Applying \cite[Lemma 3.3.7]{RD} with $x = \bt^\tra X$, we get that if $\E [ \| X \|^n ] < \infty$, then
\begin{align*}
\left| \E \re^{i\bt^\tra X} - \sum_{m=0}^n \E \frac{(i\bt^\tra X)^m}{m!} \right| &\le \E \left| \re^{i\bt^\tra X} - \sum_{m=0}^n \frac{(i\bt^\tra X)^m}{m!} \right| \\
&\le \E \min \left( \frac{| \bt^\tra X |^{n+1}}{(n+1)!} , \frac{2 | \bt^\tra X |^{n} }{n!} \right).
\end{align*}
Taking $n=2$ and rearranging, we get equation~\eqref{eq:estchf}, and 
$|W(\bt)| \le \E Z(t)$,
where $Z(\bt) = \min \{\|\bt\| \|X \|^3, \|X \|^2 \}$. Now $|Z(\bt)| \le \|X \|^2$ and $\E[ \|X \|^2 ] < \infty$. Also we have $|Z(\bt)| \le \|\bt\| \|X \|^3 \to 0$ a.s.\ as $\|\bt\| \to 0$. So the dominated convergence theorem implies that $\E Z(\bt) \to 0$ as $\|\bt \| \to 0$. 
\end{proof}

We collect some facts about lattice distributions: 
for reference see \cite[Ch.~5]{BR} and \cite[\S 7]{spitzer}. 
Let
\[ \cH := \{  H   : \Pr ( X \in \bb + H \Z^d ) =1 \text{ for some } \bb \in \R^d \} .\]
If $X$ has a lattice distribution, then $\cH$ is nonempty, and if $X$
is non-degenerate then any $H \in \cH$ has $| \det H | >0$. (Here and elsewhere, `non-degenerate' means not supported on any $(d-1)$-dimensional hyperplane.)
Let $K := \{ | \det H | : H \in \cH \}$. The next result gives an upper bound on $h \in K$;
note that this bound
is sharp in both of the examples in Section~\ref{sec:examples}.

\begin{lemma}
\label{lem:h-bounded}
Suppose that $X$ has a non-degenerate lattice distribution. Then $K \subseteq (0,\infty)$ is bounded, and $\inf K = 0$.
\end{lemma}
\begin{proof}
Since $X$ has a non-degenerate lattice distribution, we have that (i) $\cH$ is non-empty and $|\det H| > 0$ for all $H \in \cH$;
and (ii) there exists $\cX := \{ \bx_0, \bx_1, \ldots, \bx_d \}$ such that
$\bx_0, \ldots, \bx_d$ are affinely
independent, and $\Pr ( X = \bx_i ) >0$ for each $i$. 
Statement (i) shows that $K \subseteq (0,\infty)$ is nonempty,
and statement (ii) shows that $K$ is bounded. 
Indeed, for any $H \in \cH$ we have that there exists $\bb$ such that $\cX \subset  \bb + H \Z^d$, i.e.,
$H^{-1} ( \cX - \bb ) \subset  \Z^d$.
For $i \in \{1,\ldots, d\}$ let $\lambda_i = \bx_i - \bx_0$.
Then the linearly independent vectors $\lambda_1, \ldots, \lambda_d$ define a parallelepiped $P$
with volume $| \det \Lambda | \in (0,\infty)$, where $\Lambda$ denotes the $d \times d$ matrix whose columns are $\lambda_1, \ldots, \lambda_d$.
Since $H^{-1} ( \cX - \bb )$ are points of $\Z^d$, we have that all the vertices of the parallelepiped $P':= H^{-1} (\bx_0 + P - \bb )$
are points of $\Z^d$. Now $P'$ has volume $h^{-1} | \det \Lambda | > 0$, but, as a parallelepiped of positive volume
whose vertices are in $\Z^d$, must have volume at least 1. Thus $h^{-1} | \det \Lambda | \geq 1$, i.e., $h \leq | \det \Lambda | < \infty$.
Also, we see that if $H \in \cH$, then $H/2 \in \cH$ as well, so if $h \in K$ then $h/2^d \in K$ too. 
\end{proof}

Define
$U := \{ \bt \in \R^d : | \varphi (\bt ) | = 1 \}$.
Given an invertible $d$ by $d$ matrix $H$, set $S_H := 2 \pi ( H^\tra)^{-1} \Z^d$.
The next result shows that if $H \in \cH$, then $S_H \subseteq U$.

\begin{lemma}
\label{lem:property}
Suppose that $H \in \cH$. Then 
$|\varphi (\bu )| = 1$ for all $\bu \in S_H$.
\end{lemma}
\begin{proof}
First observe that the norm of the characteristic function is invariant under translation by any vector of the form of $2\pi (H^\tra)^{-1} \bk$ with $\bk \in \Z^d$. 
To see this, note that for any $\bk \in \Z^d$, 
\[
\left|\varphi(\bt+2\pi (H^\tra)^{-1}\bk)\right|=\left| \E \left[ \re^{i\bt^\tra X}  \cdot \re^{ 2 \pi i \bk^\tra H^{-1} X} \right] \right|.
\]
Since $H \in \cH$, we may write
 $X =\bb + H W$, where $\bb \in \R^d$ is constant and $W \in \Z^d$. Hence 
\[
\left|\varphi(\bt+2\pi (H^\tra)^{-1}\bk)\right|
=\left|\re^{ 2 \pi i \bk^\tra H^{-1} \bb}\right| \cdot \left| \E \left[ \re^{i\bt^\tra X} \cdot \re^{2 \pi i \bk^\tra W } \right] \right|  ,\]
because $\bk^\tra H^{-1} \bb$ is a non-random scalar.
Then, since $|\exp\{ 2 \pi i \bk^\tra H^{-1} \bb \}|=1$
and $\bk^\tra W \in \Z$, so that 
$\exp\{2 \pi i \bk^\tra W \}=1$, it follows that for any $\bk \in \Z^d$,
\begin{equation}
\label{eq:periodic}
\left|\varphi(\bt+2\pi (H^\tra)^{-1}\bk)\right| = \left|\varphi(\bt)\right|.
\end{equation}
In particular, the case $\bt = \0$ of~\eqref{eq:periodic} 
shows that $|\varphi( \bu ) | = 1$ if $\bu \in S_H$.
\end{proof}

If $\Pr ( X \in \bb + H \Z^d) = 1$ and $\Pr (X = \bx ) >0$, then
$\bx - \bb \in H \Z^d$ so that $\bx + H \Z^d = \bb + H\Z^d$,
and so if $H \in \cH$ then $\Pr (X \in \bx + H \Z^d ) = 1$
for any $\bx$ with $\Pr (X = \bx ) >0$.

Lemma~21.4 of~\cite{BR} shows that there is a unique minimal subgroup $L$ of $\R^d$ such that
$\Pr ( X \in \bx + L ) =1$ for any $\bx$ with $\Pr ( X = \bx ) >0$ and
if $H \in \cH$ then
$L \subseteq H \Z^d$. Moreover, the discrete subgroup $L$
is generated by $\{ \xi : \Pr ( X = \bx + \xi ) > 0 \}$ for any given
$\bx$ with $\Pr ( X = \bx ) > 0$. We have $L= H_0 \Z^d$ for some (not necessarily unique) $H_0 \in \cH$;
let $\cH_0 := \{ H \in \cH : L = H \Z^d \}$.

The next result gives equivalent formulations of the fundamental assumption~\eqref{ass:basicd2}.
For $\rho >0$, define $S_H(\rho) := \cup_{\by \in S_H} B(\by;\rho)$, 
where $B(\by; \rho)$ is the open Euclidean ball of radius $\rho$ centred at $\by \in \R^d$.

\begin{lemma}
\label{lem:equivalent}
Suppose that $X$ is non-degenerate and $H \in \cH$.
The following are equivalent.
\begin{itemize}
\item[(i)] $H \in \cH_0$.
\item[(ii)] $| \det H |$ is the maximal element of $K$.
\item[(iii)] $S_H = U$.
\end{itemize}
Moreover, if any one of these conditions holds then, for any $\rho >0$, there exists a positive constant $c_{\rho}$ such that 
\[
\left| \varphi(\bu) \right| \le \re^{-c_{\rho}}, \text{ for any } \bu \notin S_H(\rho).
\]
\end{lemma}
\begin{proof}
Suppose that $H_0 \in \cH_0$ and $H \in \cH$.
Let $h_0 = | \det H_0|$ and $h = | \det H|$.
Then, by minimality, $H_0 \Z^d \subseteq H \Z^d$, i.e., $H^{-1} H_0 \Z^d \subseteq \Z^d$.
Thus $H^{-1} H_0 [0,1]^d$ is a parallelepiped whose vertices are all in $\Z^d$,
and necessarily this parallelepiped has volume at least 1. Hence $h_0 / h \geq 1$, i.e., $h \leq h_0$. 
Thus if $H \in \cH_0$ then $| \det H |$ is maximal. 
On the other hand, suppose $H \in \cH \setminus \cH_0$ and $H_0 \in \cH_0$.
Then $H_0 \Z^d \subset H \Z^d$ are not equal, so there is some $\bx \in H \Z^d$ with $\bx \notin H_0 \Z^d$.
Thus for $\by = H^{-1} \bx \in \Z^d$, we have that $H^{-1} H_0 \Z^d \subset \Z^d$ with $\by \notin H^{-1} H_0 \Z^d$.
For $\bz \in \Z^d$ we have $\by = H^{-1} H_0 ( \bz + \alpha )$ where $\alpha \in [0,1]^d$ is not a vertex;
but then $\by - H^{-1} H_0 \bz \in \Z^d$ as well. Thus $\beta = H^{-1} H_0 \alpha$ is a point of $\Z^d$
contained in the parallelepiped $P = H^{-1} H_0 [0,1]^d$, and moreover all the vertices of $P$ are in $\Z^d$,
and $\beta$ is not a vertex. Hence
the parallelepiped $P$ has volume strictly greater than 1 (see \cite[p.~69]{spitzer}), and so $h_0 / h > 1$.
Thus if $H \notin \cH_0$ then $| \det H |$ is not maximal. 
 Thus (i) and (ii) are equivalent.

We show that (i) implies (iii).
For $H \in \cH$ set
\begin{align*} R_H & := \{ \bt \in \R^d : \bx^\tra \bt \in 2 \pi \Z \text{ for all } \bx \in H \Z^d \}\\
& = \{ \bt \in \R^d : \bz^\tra H^\tra \bt \in 2 \pi \Z \text{ for all } \bz \in \Z^d \} . \end{align*}
It follows that
\[ R_H =  2 \pi ( H^\tra )^{-1} \{ \by \in \R^d : \bz^\tra \by \in \Z \text{ for all } \bz \in \Z^d \} =  
 2 \pi ( H^\tra )^{-1}  \Z^d = S_H .\]
So $R_H = S_H$ for any $H \in \cH$ with $| \det H | >0$. Moreover,
Lemma~21.6 of~\cite{BR} shows that $R_H = U$ if $H \Z^d$ is minimal.
Thus (i) implies (iii).

Next we show that (iii) implies (ii).
Let $h_\star := \sup K$, which, by Lemma~\ref{lem:h-bounded} is finite and positive.
Suppose that $H \in \cH$ with $|\det H| = h \in (0,h_\star)$.
Then for any $\eps >0$ sufficiently small, we can find $H_1 \in \cH$ with $| \det H_1 | = h_1 \in ( h , h_\star]$ such that $h_1 >  ( 1+2 \eps) h$
and $h_1 > (1-\eps) h_\star$. Let $S = 2\pi (H^\tra)^{-1} \Z^d$ and $S_1 = 2\pi (H_1^\tra)^{-1} \Z^d$.

Consider $\bx$ with $\Pr (X = \bx ) >0$. Then there exist $\bb, \bb_1 \in \R^d$ (not depending on $\bx$) and $\bz,\bz_1 \in \Z^d$ (depending on $\bx$)
such that 
\[ \bx = \bb + H \bz  = \bb_1 + H_1 \bz_1 ,\]
and hence 
\begin{equation}
\label{calH1}
\bz = H^{-1} (\bb_1 - \bb) + H^{-1} H_1 \bz_1. 
\end{equation}
Take $\bs = 2\pi(H_1^\tra)^{-1} \bz_1 \in S_1$.
Assume, for the purpose of deriving a contradiction,
 that $S_1 \subseteq S$. Then $\bs \in S$, i.e., there exists $\bz_2 \in \Z^d$ such that
\[ \bs = 2\pi(H_1^\tra)^{-1} \bz_1 = 2\pi(H^\tra)^{-1} \bz_2.\]
Together with~\eqref{calH1}, this implies that 
\[ \bz = H^{-1} H_1 H_1^{\tra} (H^\tra)^{-1} \bz_2 + H^{-1}(\bb_1 - \bb) .\]
It follows that 
\begin{align*}
\bx  = \bb + H \bz  = \bb_1 + H_1 H_1^{\tra} (H^\tra)^{-1} \bz_2.
\end{align*}
Now if we take $\bb_2= \bb_1$ and $H_2 = H_1 H_1^{\tra} (H^\tra)^{-1}$, 
we have shown that every $\bx$ for which $\Pr ( X = \bx ) >0$ has $\bx \in \bb_2 + H_2 \Z^d$, i.e.,
$H_2 \in \cH$. But
\begin{align*}
\left|\det H_2 \right| & = \left|\det H_1 \right| \left|\det H_1^{\tra} \right| \left|\det (H^\tra)^{-1} \right| = \frac{h_1^2}{h} \\
& > (1+2\eps) (1-\eps) h_\star > h_\star ,\end{align*}
for $\eps$ sufficiently small,
which contradicts the definition of $h_\star$. Thus there exists some $\bx \in S_1$ with $\bx \notin S$.

From Lemma~\ref{lem:property}, we have $S_1 \subseteq U$; hence there is some $\bx \in U$ with $\bx \notin S$. In other words,
we have shown that if $h \in (0,h_\star)$ then $S \neq U$. Thus if we assume that $S = U$, the only possibility is $h = h_\star \in K$.
Thus (iii) implies (ii).

To prove the final statement in the lemma, we may suppose that (iii) holds. Then $| \varphi ( \bu ) | < 1$ if $\bu \notin S_H$.
To finish the proof of the lemma, it suffices to show that $\sup_{\bu \notin S_H(\rho)} |\varphi(\bu)| <1$.
But, by the periodicity of $|\varphi ( \bu ) |$ from~\eqref{eq:periodic}, we have
$\sup_{\bu \notin S_H(\rho)} | \varphi (\bu) | = \sup_{\bu \in T_H(\rho)} | \varphi (\bu) |$
where $T_H(\rho) := 2 \pi ( H^\tra )^{-1} [ -\frac{1}{2}, \frac{1}{2} ]^d \setminus B ( \0 ; \rho )$.
Suppose that $\sup_{\bu \in T_H(\rho)} | \varphi (\bu) | = 1$; then by the continuity of $|\varphi(\bu)|$, the supremum is attained at a point $\bu$ in the compact set $T_H(\rho)$, 
contradicting the fact that $|\varphi (\bu)| < 1$ for all $\bu \notin S_H$. Hence $\sup_{\bu \in T_H(\rho)} | \varphi (\bu) | < 1$, and the proof is completed.
\end{proof}

\section{Proofs of Propositions~\ref{prop:LLN} and~\ref{prop:CLT}}
\label{sec:appendix}

\begin{proof}[Proof of Proposition~\ref{prop:LLN}.]
By the strong law for $S_n$, we have that
for any $\eps>0$ there  exists $N_\eps$ with $\Pr (N_\eps < \infty) =1$ such that $\| S_n - n\bmu \| \le n \eps$ for all $n \ge N_\eps$. Then, by the triangle inequality,
\begin{align*}
\left\| G_n - (n+1) (\bmu/2)  \right\| & = \frac{1}{n} \left\|\sum_{i=1}^n ( S_i - i \bmu ) \right\| \\
&\le \frac{1}{n}  \sum_{i=1}^{N_\eps} \| S_i - i \bmu \|  + \frac{1}{n}   \sum_{i=N_\eps}^n \| S_i - i \bmu \|  \\
& \leq \frac{1}{n}  \sum_{i=1}^{N_\eps} \| S_i - i \bmu \|  + \frac{1}{n} \sum_{i=1}^n i \eps .
\end{align*}
It follows that 
\[ \limsup_{n \to \infty} n^{-1} \left\| G_n - (n+1) (\bmu/2)  \right\| \leq \eps /2 ,\]
and since $\eps >0$ was arbitrary we get the result.
\end{proof}

\begin{proof}[Proof of Proposition~\ref{prop:CLT}.]
For any unit vector $\be \in \R^d$, 
$\be \cdot G_n$ is the centre-of-mass associated with the one-dimensional random walk
with increments $\be \cdot X_i$; thus, by the Cramer--Wold device (see e.g~\cite[Theorem~3.9.5]{RD}), it suffices to 
establish the central limit theorem for $d=1$.

So take $d=1$ and write $\bmu = \mu$, $M = \sigma^2 \in (0,\infty)$.
It follows from~\eqref{eq:weighted-sum} that for fixed $n$,
 $G_n$ has the same distribution as
\[ G_n' := \sum_{i=1}^n \left( \frac{i}{n} \right) X_i .\]
It thus suffices to show that $n^{-1/2} ( G_n' - \frac{n}{2} \mu )$ converges in distribution
to $\cN_1 (0, \sigma^2/3)$. We show that this follows from~\cite[Corollary~8.4.1]{AB}.
Define $T_{n,i} := \frac{i}{n^{3/2}} ( X_i - \mu)$, so that
\[ \sum_{i=1}^n T_{n,i} - n^{-1/2} \left( G_n' - \frac{n}{2} \mu \right) \to 0, \as \]
Then
\[ \sum_{i=1}^n \Var (T_{n,i}) = \sum_{i=1}^n \frac{i^2}{n^3} \sigma^2 \to \frac{\sigma^2}{3}. \]
It remains to verify the Lindeberg condition for triangular arrays: for every $\eps>0$,
\[ 
\lim_{n \to \infty} \sum_{i=1}^n \E \left[ T_{n,i}^2 \1 {|T_{n,i}| > \eps }  \right] = 0 .
\] 
But we have that 
\begin{align*}
\sum_{i=1}^n \E \left[T_{n,i}^2 \1{|T_{n,i}| > \eps}  \right] & \le \sum_{i=1}^n \E \left[T_{n,n}^2 \1{|T_{n,n}| > \eps}  \right] \\
&= \sum_{i=1}^n \frac{1}{n} \E \left[(X-\mu)^2 \1{|X -\mu| > \eps\sqrt{n}}  \right] \\
&= \E \left[(X -\mu)^2 \1{|X-\mu| > \eps\sqrt{n}}  \right].
\end{align*}
Now   $(X-\mu)^2 \1{|X -\mu| > \eps \sqrt{n}} \to 0$ a.s.~as $n \to \infty$ and  $|(X -\mu)^2 \1{|X -\mu| > \eps \sqrt{n}}| \le (X-\mu)^2$ which has $\E [ (X-\mu)^2 ] < \infty$. Thus the dominated convergence theorem yields $\E[(X-\mu)^2 \1{|X-\mu| > \eps \sqrt{n}}] \to 0$ as $n \to \infty$ and the Lindeberg condition is verified,
and~\cite[Corollary~8.4.1]{AB} shows that $\sum_{i=1}^n T_{n,i}$ converges in distribution to  $\cN_1 (0, \sigma^2/3)$.
\end{proof}

\section*{Acknowledgements}

The authors are grateful to Ostap Hryniv and Mikhail Menshikov
for fruitful discussions on the topic of this paper, to two anonymous referees
for their comments, and to Francesco Caravenna for bringing references~\cite{CD,DKW} to our attention.

\end{document}